\documentclass[a4paper,10pt,twoside]{article}

\usepackage{cmap} 

\usepackage[french,ngerman,english]{babel}
\usepackage[OT2,T1]{fontenc}
\usepackage[latin1]{inputenc}
\usepackage{amsmath,amsfonts,amssymb,amsthm,amscd}
\usepackage{mathrsfs}
\usepackage[all,cmtip]{xy}
\usepackage{tikz}
\usetikzlibrary{matrix}
\usepackage{pdflscape}

\usepackage[babel]{csquotes}

\usepackage[pdftex,
	hyperindex=true, bookmarks=true, bookmarksopen=true, bookmarksnumbered=true,
    pdfauthor={Dr. Timo Keller},
    pdftitle={On the Tate-Shafarevich group of Abelian schemes over higher dimensional bases over finite fields},
    pdfsubject={},
    pdfkeywords={Abelian varieties, Tate-Shafarevich groups, Brauer groups, higher dimensional function fields},
    linkcolor=blue
]{hyperref}

\usepackage{enumitem}
\usepackage{relsize}
\usepackage{fancyhdr}

\usepackage{vmargin}
\setpapersize{A4}
\setmargins{2.1cm}{1.5cm} 
           {17cm}{23.5cm}   
           {2em}{2em}     
           {0pt}{0em}

\usepackage{ae,aecompl}
	\usepackage[step=1,stretch=30,shrink=30,selected=true]{microtype}

\DeclareTextFontCommand\textmathbf{\bfseries\boldmath} 

\newcommand{\N} {\ensuremath{\mathbf{N}}}
\newcommand{\Z} {\ensuremath{\mathbf{Z}}}
\newcommand{\Q} {\ensuremath{\mathbf{Q}}}
\newcommand{\R} {\ensuremath{\mathrm{R}}}
\newcommand{\RR} {\ensuremath{\mathbf{R}}}
\newcommand{\G} {\ensuremath{\mathbf{G}}}
\newcommand{\A} {\ensuremath{\mathbf{A}}}
\newcommand{\Ccal} {\ensuremath{\mathscr{C}}}
\newcommand{\Ocal} {\ensuremath{\mathscr{O}}}
\newcommand{\Ecal} {\ensuremath{\mathscr{E}}}
\newcommand{\Lcal} {\ensuremath{\mathscr{L}}}

\newcommand{\F} {\ensuremath{\mathbf{F}}}
\newcommand{\p} {\ensuremath{\mathfrak{p}}}

\renewcommand{\H} {\ensuremath{\mathrm{H}}}
\newcommand{\PPic} {\ensuremath{\mathbf{Pic}}}

\renewcommand{\epsilon}{\varepsilon}
\renewcommand{\theta}{\vartheta}
\renewcommand{\phi}{\varphi}

\newcommand*\defn[1]{\textup{\textmathbf{#1}}}

\providecommand{\pfister}[1]{\langle\kern-0.2em\langle#1\rangle\kern-0.2em\rangle}

\providecommand{\iso}{\ensuremath{\cong}}

\providecommand{\isoto}{\xrightarrow{~\sim~}}

\DeclareMathOperator{\im}{im}
\DeclareMathOperator{\infl}{inf}

\DeclareMathOperator{\End}{End}
\DeclareMathOperator{\Hom}{Hom}

\DeclareMathOperator{\Gal}{Gal}

\DeclareMathOperator{\id}{id}

\DeclareMathOperator{\codim}{codim}
\DeclareMathOperator{\PGL}{PGL}
\DeclareMathOperator{\colim}{\varinjlim}
\DeclareMathOperator{\res}{res}

\DeclareMathOperator{\Char}{char}
\DeclareMathOperator{\Br}{Br}
\DeclareMathOperator{\Ext}{Ext}

\DeclareMathOperator{\Spec}{Spec}
\DeclareMathOperator{\Spf}{Spf}

\DeclareMathOperator{\pr}{pr}

\DeclareMathOperator{\GL}{GL}
\DeclareMathOperator{\rk}{rk}

\DeclareMathOperator{\Pic}{Pic}
\DeclareMathOperator{\coker}{coker}
\DeclareMathOperator{\Quot}{Quot}

\DeclareMathOperator{\Tors}{Tors}

\providecommand{\Acal}{\ensuremath{\mathscr{A}}}

\providecommand{\Ecal}{\ensuremath{\mathscr{E}}}
\providecommand{\Fcal}{\ensuremath{\mathscr{F}}}

\providecommand{\Hcal}{\ensuremath{\mathscr{H}}}
\providecommand{\Ical}{\ensuremath{\mathscr{I}}}

\providecommand{\Lcal}{\ensuremath{\mathscr{L}}}

\providecommand{\Ocal}{\ensuremath{\mathscr{O}}}

\providecommand{\EExt}{\ensuremath{\mathscr{E}\mathrm{xt}}}
\providecommand{\HHom}{\ensuremath{\mathscr{H}\mathrm{om}}}

\providecommand{\Het}{\ensuremath{\H_\mathrm{\acute{e}t}}}

\providecommand{\piet}{\ensuremath{\pi_1^\mathrm{\acute{e}t}}}

\newcommand{\Div} {\ensuremath{\mathrm{Div}}}

\newcommand{\cd} {\ensuremath{\mathrm{cd}}}

\providecommand{\Zar}{\ensuremath{\mathrm{Zar}}}
\providecommand{\et}{\ensuremath{\mathrm{\acute{e}t}}}
\providecommand{\fppf}{\ensuremath{\mathrm{fppf}}}

\DeclareSymbolFont{cyrletters}{OT2}{wncyr}{m}{n}
\DeclareMathSymbol{\Sha}{\mathalpha}{cyrletters}{"58}

\newtheorem{theorem}{Theorem}[section]
\newtheorem{lemma}[theorem]{Lemma}
\newtheorem{corollary}[theorem]{Corollary}
\newtheorem{definition}[theorem]{Definition}
\newtheorem{proposition}[theorem]{Proposition}

\theoremstyle{definition}

\newtheorem{remark}[theorem]{Remark}
\usepackage{cleveref}

\crefname{theorem}{Theorem}{Theorems}
\crefname{lemma}{Lemma}{Lemmata}
\crefname{corollary}{Corollary}{Corollaries}
\crefname{proposition}{Proposition}{Propositions}
\crefname{definition}{Definition}{Definitions}
\crefname{conjecture}{Conjecture}{Conjectures}
\crefname{example}{Example}{Examples}
\crefname{algorithm}{Algorithm}{Algorithms}
\crefname{remark}{Remark}{Remarks}

\numberwithin{equation}{section}

\addto\captionsngerman{%
	\renewcommand{\refname}{Bibliography}%
}

\allowdisplaybreaks[1]

\newcounter{marginlabelcounter}

\begin{document}
\nonfrenchspacing


\pagestyle{fancy}

\title{On the Tate-Shafarevich group of Abelian schemes over higher dimensional bases over finite fields}
\author{Timo Keller}
\maketitle\thispagestyle{empty}
\begin{abstract}
We study analogues for the Tate-Shafarevich group for Abelian schemes with everywhere good reduction over higher dimensional bases over finite fields.
\end{abstract}

\markright{}


\section{Introduction}

The Tate-Shafarevich group $\Sha(\Acal/X)$ of an Abelian scheme $\Acal$ over a scheme $X$ is of great importance for the arithmetic of $\Acal$.  It classifies everywhere locally trivial $\Acal$-torsors.  If $X$ is the spectrum of the ring of integers of a number field or a smooth projective geometrically connected curve over a finite field, so that the function field $K = K(X)$ of $X$
 is a global field, one has
\[
    \Sha(\Acal/X) = \ker\Big(\H^1(K,A) \to \prod_v\H^1(K_v,A)\Big),
\]
where $v$ runs over all places of $K$ and $K_v$ is the completion of $K$ with respect to $v$.  The aim of this article is to generalise this definition to the case of a \emph{higher dimensional basis} $X/\F_q$ and prove some properties for this group.

In section~3, we show that an Abelian scheme $\Acal/X$ over $X$ regular, Noetherian, integral and separated satisfies the Néron mapping property, namely that $\Acal = g_*g^*\Acal$ on the smooth site of $X$, where $g: \{\eta\} \hookrightarrow X$ denotes the inclusion of the generic point.  In section~4.1, we define the Tate-Shafarevich group for Abelian schemes $\Acal$ over higher dimensional bases $X$ as $\Sha(\Acal/X) := \Het^1(X,\Acal)$ and show:
\[
	\H^1(X, \Acal) = \ker\Big(\H^1(K, \Acal) \to \prod_{x \in S} \H^1(K_x^{nr}, \Acal) \Big),
\]
where $K_x^{nr} = \Quot(\Ocal_{X,x}^{sh})$, and $S$ is either (a) the set of all points of $X$, or (b) the set $|X|$ of all closed points of $X$, or (c) the set $X^{(1)}$ of all codimension-$1$ points of $X$, and $\Acal = \PPic^0_{\Ccal/X}$ for a relative curve $\Ccal/X$ with everywhere good reduction admitting a section, and $X$ is a variety over a finitely generated field.  Here, one can replace $K_x^{nr}$ by $K_x^h = \Quot(\Ocal_{X,x}^h)$ if $\kappa(x)$ is finite, and $K_x^{nr}$ and $K_x^h$ by $\Quot(\hat{\Ocal}_{X,x}^{sh})$ and $\Quot(\hat{\Ocal}_{X,x}^h)$, respectively, if $x \in X^{(1)}$.  The obvious conjecture is that the Tate-Shafarevich group $\Sha(\Acal/X)$ is finite.

In section~2 and~4.2, for a (split) relative curve $\Ccal/X$ we relate the Brauer groups of $X$ and $\Ccal$ to the Tate-Shafarevich group of $\PPic^0_{\Ccal/X}$:  There is an exact sequence
\[
    0 \to \Br(X) \stackrel{\pi^*}{\to} \Br(\Ccal) \to \Sha(\PPic^0_{\Ccal/X}/X) \to 0.
\]
This generalises results of Artin and Tate~\cite{Artin-Tate}.

In section~4.3, we show that finiteness of an $\ell$-primary component of the Tate-Shafarevich group descents under generically étale $\ell'$-alterations.  This is used in~\cite{KellerHeight}, Theorem~4.18 and Remark~4.19 to prove the finiteness of the Tate-Shafarevich group for isotrivial Abelian schemes over finite fields under mild conditions.  In section~4.4, we show that finiteness of an $\ell$-primary component of the Tate-Shafarevich group is invariant under étale isogenies.  In section~4.5, we construct a Cassels-Tate pairing $\Sha(\Acal/X)[\ell^\infty] \times \Sha(\Acal^\vee/X)[\ell^\infty] \to \Q_\ell/\Z_\ell$ in some cases.  In~\cite{KellerHeight}, we also use the results of this article for studying the Birch-Swinnerton-Dyer conjecture for Abelian schemes over higher dimensional bases over finite fields.

\paragraph{Notation.}
All rings are commutative with 1.  For an Abelian group $A$, let $\Tors{A}$ be the torsion subgroup of $A$.  Denote the cokernel of $A \stackrel{n}{\to} A$ by $A/n$ and its kernel by $A[n]$, and the $p$-primary subgroup $\colim_n A[p^n]$ by $A[p^\infty]$ for any prime $p$.  Canonical isomorphisms are often denoted by ``$=$''.  If not stated otherwise, all cohomology groups are taken with respect to the étale topology.  By the $\ell$-cohomological dimension $\cd_\ell(X)$ of a scheme $X$ we mean the smallest integer $n$ (or $\infty$) such that $\H^q(X,\Fcal) = 0$ for all $i > n$ and all $\ell$-torsion sheaves $\Fcal$.  The ind-étale sheaf $\mu_{\ell^\infty}$ is defined as $\colim_n \mu_{\ell^n}$ with $\mu_{\ell^n}$ the étale sheaf of all $\ell^n$-th roots of unity, so $\H^q(X,\mu_{\ell^\infty}) = \colim_n\H^q(X,\mu_{\ell^n})$.  We denote Pontryagin duals, duals of $R$-modules and Abelian schemes by $(-)^\vee$.  It should be clear from the context which one is meant.  Pontryagin dual of a locally compact Abelian group $A$ is defined as $A^\vee = \Hom_{cont}(A,\RR/\Z)$ endowed with the compact-open topology.  The Henselisation of a local ring $A$ is denoted by $A^h$ and the strict Henselisation by $A^{sh}$.  By a finitely generated field we mean a field finitely generated over its prime field.  If $A$ is an integral domain, we denote by $\Quot(A)$ its quotient field.



\section{Brauer groups, Picard groups and cohomology of $\G_m$}
The main new technical result of this section is~\cref{thm:H2bijective}, which is needed to prove the main theorem~\cref{thm:RqGm=0q>1} for the $p$-torsion (the prime-to-$p$-torsion can be treated much easier) and~\cref{cor:LangeExakteSequenzMitGm} and~\cref{cor:BrauerUndShaExakteSequenz}.

We collect some well-known results on the cohomology of $\G_m$.
\begin{lemma} \label[lemma]{thm:KurzeExakteSequenzAusKummerSequenz}
Let $X$ be a scheme and $\ell$ a prime invertible on $X$. Then there are exact sequences
\[
	0 \to \H^{i-1}(X,\G_m) \otimes_\Z \Q_\ell/\Z_\ell \to \H^i(X,\mu_{\ell^\infty}) \to \H^i(X,\G_m)[\ell^\infty] \to 0
\]
for each $i \geq 1$.
\end{lemma}
\begin{proof}
This follows from the long exact sequence induced by the Kummer sequence.
\end{proof}

\begin{definition}
A \defn{variety} is a separated scheme of finite type over over a field $k$.
\end{definition}

Recall that the \defn{Brauer group} $\Br(X)$ of a scheme $X$ is the group of equivalence classes of Azumaya algebras on $X$ (see e.\,g.~\cite{MilneÉtaleCohomology}, IV.2, p.~140\,ff.).

\begin{definition} \label[definition]{def:BrauerPrimeH2Tors}
$\Br'(X) := \Tors\H^2(X,\G_m)$ is called the \defn{cohomological Brauer group}.
\end{definition}

\begin{theorem} \label{thm:BrStrichGleichH2} \label{thm:BrauerGabber}
(a) There is an injection $\Br(X) \hookrightarrow \Br'(X)$.

(b) One has $\Br(X) = \Br'(X)$ if $X$ has an ample invertible sheaf.

(c) $\Br'(X) = \H^2(X,\G_m)$ if $X$ is a regular integral quasi-compact scheme.
\end{theorem}
\begin{proof}
(a) See \cite{MilneÉtaleCohomology}, p.~142, Theorem~2.5.  (b) See \cite{GabberBrauerGroup}.  (c) See \cite{MilneÉtaleCohomology}, p.~106\,f., Example~2.22.
\end{proof}

\begin{corollary} \label[corollary]{cor:BrFuerVarietyOverFiniteField}
Let $X/k$ be a regular quasi-projective geometrically connected variety.  Then $\Br(X) = \Br'(X) = \H^2(X,\G_m)$.
\end{corollary}

\begin{theorem} \label{thm:LichtenbaumBrauer}
Let $X$ be a smooth projective geometrically connected variety over a finite field $k = \F_q$, $q = p^n$.

(a) $\H^i(X,\G_m)$ is torsion for $i \neq 1$, finite for $i \neq 1,2,3$ and trivial for $i > 2\dim(X) + 1$.  For $\H^i(X,\G_m)$ being torsion for $i > 1$, $k$ can be any field.

(b) For $\ell \neq p$ and $i = 2,3$, one has $\H^i(X,\G_m)[\ell^\infty] \iso (\Q_\ell/\Z_\ell)^{\rho_{i,\ell}} \oplus C_{i,\ell}$, where $C_{i,\ell}$ is finite and trivial for all but finitely many $\ell$, and $\rho_{i,\ell}$ a non-negative integer.

(c)  Let $\ell \neq p$ be prime.  Then one has
\[
    \H^i(X,\mu_{\ell^\infty}) \isoto \H^i(X,\G_m)[\ell^\infty]
\]
for $i \neq 2$.
\end{theorem}
\begin{proof}
(a) and (b):  See \cite{Lichtenbaums=1}, p.~180, Proposition~2.1\,a)--c), f) (the assertion about $\H^i(X,\G_m)$ being torsion for $i > 1$ follows from the proof given there).  (c) follows from~\cref{thm:KurzeExakteSequenzAusKummerSequenz} and (a).
\end{proof}

\begin{lemma} \label[lemma]{lemma:AzumayaTrivialised}
Let $X$ be a quasi-compact scheme, quasi-projective over an affine scheme.  Assume $\underline{A} \in \check{\H}^1(X,\PGL_n)$ is an Azumaya algebra trivialised by $\underline{A} \iso \underline{\End}(\underline{V})$ with $\underline{V}$ a locally free $\Ocal_X$-sheaf of rank $n$.  Then every other such $\underline{V}'$ differs from $\underline{V}$ by tensoring with an invertible sheaf.
\end{lemma}
For the definition of \v{C}ech non-Abelian étale cohomology see~\cite{MilneÉtaleCohomology}, p.~120\,ff., III.4.
\begin{proof}
Note that $\underline{A}$ corresponds to an element of $\check{\H}^1(X,\PGL_n)$ and $\underline{V}$ corresponds to an element in $\H^1(X,\GL_n)$.  For $n\in \N$, consider the central extension of étale sheaves on $X$ (see~\cite{MilneÉtaleCohomology}, p.~146)
\[
    1 \to \G_m \to \GL_n \to \PGL_n \to 1.
\]
By~\cite{MilneÉtaleCohomology}, p.~143, Step~3, this induces a long exact sequence in (\v{C}ech) cohomology of pointed sets
\[
    \Pic(X) = \check{\H}^1(X,\G_m) \stackrel{g}{\to} \check{\H}^1(X,\GL_n) \stackrel{h}{\to} \check{\H}^1(X,\PGL_n) \stackrel{f}{\to} \check{\H}^2(X,\G_m).
\]
Note that by assumption and~\cite{MilneÉtaleCohomology}, p.~104, Theorem~III.2.17, $\check{\H}^1(X,\G_m) = \H^1(X,\G_m) = \Pic(X)$ and $\check{\H}^2(X,\G_m) = \H^2(X,\G_m)$.  Further, $\Br(X) = \Br'(X)$ since a scheme quasi-compact and quasi-projective over an affine scheme has an ample line bundle (\cite{Liu2006}, p.~171, Corollary~5.1.36), so~\cref{thm:BrauerGabber} applies and $\Br'(X) \hookrightarrow \H^2(X,\G_m)$.  Since $\underline{A}$ is an Azumaya algebra, $f(\underline{A}) = [\underline{A}] \in \Br(X) \hookrightarrow \H^2(X,\G_m)$.  Therefore $f$ factors through $\Br(X) \hookrightarrow \H^2(X,\G_m)$.  

Assume the Azumaya algebra $\underline{A} \in \H^1(X,\PGL_n)$ lies in the kernel of $f$, i.\,e.\ there is a $\underline{V}$ such that $\underline{A} \iso \underline{\End}(\underline{V})$.  Then it comes from $\underline{V} \in \H^1(X,\GL_n)$ by~\cite{MilneÉtaleCohomology}, p.~143, Step~2 ($h$ is the morphism $\underline{V} \mapsto \underline{\End}(\underline{V})$).  So, since $\G_m$ is central in $\GL_n$, by the analogue of~\cite{SerreGaloisCohomology}, p.~54, Proposition~42 for étale \v{C}ech cohomology, if $\underline{V}' \in \H^1(X,\GL_n)$ also satisfies $\underline{A} \iso \underline{\End}(\underline{V}')$, they differ by an invertible sheaf.
\end{proof}

\begin{lemma} \label[lemma]{lemma:BrauerProjlim}
Let $f: \Ccal \to X$ be a projective smooth morphism with $X = \Spec(A)$ the spectrum of a Henselisation of a variety at a regular point.  Let $X_n = \Spec(A/\mathfrak{m}^{n+1})$ with $\mathfrak{m}$ the maximal ideal of $A$ and $\Ccal_n = \Ccal \times_X X_n$ with maps $X_n \hookrightarrow X_{n+1}$ and $\Ccal_n \hookrightarrow \Ccal_{n+1}$.  Suppose the transition maps of $(\Pic(\Ccal_n))_{n \in \N}$ are surjective (in fact, the Mittag-Leffler condition would suffice).  Then the canonical homomorphism
\[
    \Br(\Ccal) \to \varprojlim_{n \in \N}\Br(\Ccal_n)
\]
is injective.
\end{lemma}
For the definition of the Mittag-Leffler condition, confer~\cite{Weibel}, p.~82, Definition~3.5.6.
\begin{proof}
Let $\underline{A}$ be an Azumaya algebra over $\Ccal$ which lies in the kernel of the map in this lemma, i.\,e.\ such that for every $n \in \N$ there is an isomorphism
\begin{equation} \label{eq:AzumayaSystem}
    u_n: \underline{A}_n \iso \underline{\End}(\underline{V}_n)
\end{equation}
with $\underline{V}_n$ a locally free $\Ocal_{\Ccal_n}$-module.  Such a $\underline{V}_n$ is uniquely determined by $\underline{A}_n$ modulo tensoring with an invertible sheaf $\underline{L}_n$ by~\cref{lemma:AzumayaTrivialised}.

Because of surjectivity of the transition maps of $(\Pic(\Ccal_n))_{n \in \N}$, one can choose the $\underline{V}_n$, $u_n$ such that the $\underline{V}_n$ and $u_n$ form a projective system:
\begin{equation} \label{eq:ProjectiveSystemOfAzumaya}
    \underline{V}_n = \underline{V}_{n+1} \otimes_{\Ocal_{\Ccal_{n+1}}} \Ocal_{\Ccal_n}
\end{equation}
and the isomorphisms~\eqref{eq:AzumayaSystem} also form a projective system:  Construct the $\underline{V}_n$, $u_n$ inductively.  Take $\underline{V}_0$ such that
\[
    \underline{A} \otimes_{\Ocal_{\Ccal}} \Ocal_{\Ccal_0} \iso \underline{A}_0 \iso \underline{\End}(\underline{V}_0).
\]
One has
\[
    \underline{A}_n = \underline{A} \otimes_{\Ocal_{\Ccal}} \Ocal_{\Ccal_n}
\]
and by~\cref{lemma:AzumayaTrivialised}, there is an invertible sheaf $\Lcal_n \in \Pic(\Ccal_n)$ such that
\[
    \underline{V}_{n+1} \otimes_{\Ocal_{\Ccal_{n+1}}} \Ocal_{\Ccal_n} \isoto \underline{V}_n \otimes_{\Ocal_{\Ccal_n}} \Lcal_n.
\]
By assumption, there is an invertible sheaf $\Lcal_{n+1} \in \Pic(\Ccal_{n+1})$ such that $\Lcal_{n+1} \otimes_{\Ocal_{\Ccal_{n+1}}} \Ocal_{\Ccal_n} \iso \Lcal_n$, so redefine $\underline{V}_{n+1}$ as $\underline{V}_{n+1} \otimes_{\Ocal_{\Ccal_{n+1}}} \Lcal_{n+1}^{-1}$.  Then~\eqref{eq:ProjectiveSystemOfAzumaya} is satisfied.

Let $\hat{X}$ be the completion of $X$, and denote by $\hat{\Ccal}, \underline{\hat{A}}, \ldots$ the base change of $\Ccal, \underline{A}, \ldots$ by $\hat{X} \to X$.

Recall that an adic Noetherian ring $A$ with defining ideal $\Ical$ is a Noetherian ring with a basis of neighbourhoods of zero of the form $\Ical^n$, $n > 0$ such that $A$ is complete and Hausdorff in this topology.  For such a ring $A$, there is the formal spectrum $\Spf(A)$ with underlying space $\Spec(A/\Ical)$.

According to~\cite{EGAIII1}, p.~150, Théorème~(5.1.4), to give a projective system $(\underline{V}_n,u_n)_{n\in\N}$ on $(\Ccal_n)_{n\in\N}$ as in~\eqref{eq:AzumayaSystem} and~\eqref{eq:ProjectiveSystemOfAzumaya} is equivalent to giving a locally free module $\underline{\hat{V}}$ on $\hat{\Ccal}$ and an isomorphism
\begin{equation} \label{eq:hatuhatA}
    \hat{u}: \underline{\hat{A}} \isoto \underline{\End}(\underline{\hat{V}}).
\end{equation}
If $X = \hat{X}$, we are done: $\underline{A} = \underline{\hat{A}}$ is trivial.

In the general case, one has to pay attention to the fact that one does not know if, with the preceding construction, $\underline{\hat{V}}$ comes from a locally free module $\underline{V}$ on $\Ccal$.  However,
there is a locally free module $\underline{\Ecal}$ on $\Ccal$ such that there exists an epimorphism
\[
    \underline{\hat{\Ecal}} \to \underline{\hat{V}}.
\]
Indeed, choosing a projective immersion for $\hat{\Ccal}$ (by projectivity of $\hat{\Ccal}/\Spec(\hat{A})$) with an ample invertible sheaf $\Ocal_{\hat{\Ccal}}(1)$, it suffices to take a direct sum of sheaves of the form $\Ocal_{\hat{\Ccal}}(-N)$, $N \gg 0$.  Now, for $N \gg 0$, there is an epimorphism $\Ocal_{\hat{\Ccal}}^{\oplus k} \twoheadrightarrow \underline{\hat{V}}(N)$ for a suitable $k \in \N$, so twisting with
$\Ocal_{\hat{\Ccal}}(-N)$ gives
\[
    \Ocal_{\hat{\Ccal}}(-N)^{\oplus k} \twoheadrightarrow \underline{\hat{V}}
\]
(``there are enough vector bundles'').  Set $\underline{\Ecal} = \Ocal_{\Ccal}(-N)^{\oplus k}$.

Now consider, for schemes $X'$ over $X$, the contravariant functor $F: (\mathrm{Sch}/X)^\circ \to (\mathrm{Set})$ given by $F(X') = $ the set of pairs $(\underline{V}', \phi')$, where $\underline{V}'$ is a quotient of a locally free module $\underline{\Ecal}' = \underline{\Ecal} \otimes_X X'$ and $\phi': \underline{A}' = \underline{A} \otimes_X X' \isoto \underline{\End}(\underline{V}')$.

Since $f$ is projective and flat, by~\cite{SGA43}, p.~133\,f., Lemme XIII 1.3, one sees that the functor $F$ is representable by a scheme, also denoted $F$, locally of finite type over $X$, hence locally of finite presentation since our schemes are Noetherian (what matters is the functor being locally of finite presentation, not its representability).  By assumption of~\cref{lemma:BrauerProjlim} and~\eqref{eq:hatuhatA}, $(\underline{\hat{V}}, \hat{u})$ is an element from $F(\hat{X})$.  By Artin approximation~\cite{ArtinApproximation}, p.~26, Theorem~(1.10) resp.\ Theorem~(1.12), $F(\hat{X}) \neq \emptyset$ implies $F(X) \neq \emptyset$:

This proves that $\underline{A}$ is isomorphic to an algebra of the form $\underline{\End}(\underline{V})$ with $\underline{V}$ locally free over $\Ccal$, so it is trivial as an element of $\Br(\Ccal)$.
\end{proof}

\begin{lemma}[deformation of units] \label[lemma]{lemma:deformationofunits}
Let $f: B \twoheadrightarrow A$ be a surjective ring homomorphism with nilpotent kernel.  If $f(b)$ is a unit, so is $b$.
\end{lemma}
\begin{proof}
Let $b \in B$ such that $f(b)$ is a unit.  Then there is $c \in A$ with $f(b)c = 1_A$.  Since $f$ is surjective, there is a $\bar{c} \in B$ such that $b\bar{c} - 1_B \in \ker(f)$, so, as a unit plus a nilpotent element is a unit, $b\bar{c}$ is a unit, so $b$ is invertible in $B$.
\end{proof}

\begin{lemma} \label[lemma]{lemma:cohomologyofnilimmersions}
Let $i: X_0 \hookrightarrow X$ be a closed immersion defined by a nilpotent ideal sheaf.  Let $\Fcal$ be an étale sheaf on $X$.  Then there is an isomorphism $\H^i(X,\Fcal) = \H^i(X_0,i^*\Fcal)$.
\end{lemma}
\begin{proof}
Since $i$ is a homeomorphism on the underlying topological spaces, the open immersion of the complement $j: X \setminus X_0 \hookrightarrow X$ is the empty set.  Hence, $\Fcal \isoto i_*i^*\Fcal$, so there is an isomorphism $\H^i(X,\Fcal) = \H^i(X,i_*i^*\Fcal)$.  Since $i_*$ is exact, the Leray spectral sequence for $i$ degenerates giving $\H^i(X,i_*i^*\Fcal) = \H^i(X_0,i^*\Fcal)$.
\end{proof}

The following is a generalisation of~\cite{GroupeDeBrauerIII}, pp.~98--104, Théorème~(3.1) from the case of $X/Y$ with $\dim{X} = 2$, $\dim{Y} = 1$ to $X/Y$ with relative dimension $1$.  One can remove the assumption $\dim{X} = 1$ if one uses Artin's approximation theorem~\cite{ArtinApproximation}, p.~26, Theorem~(1.10) resp.\ Theorem~(1.12) instead of Greenberg's theorem on p.~104, l.~4 and l.~-2, and replaces ``proper'' by ``projective'' and does some other minor modifications; also note that in our situation the Brauer group coincides with the cohomological Brauer group by~\cref{thm:LichtenbaumBrauer} and~\cref{thm:BrauerGabber}.
\begin{theorem} \label{thm:H2bijective}
Let $x$ be a closed point of a variety $V$ and $X = \Spec(\Ocal_{V,x}^h)$ be regular.  Let $f: \Ccal \to X$ be a smooth projective morphism with fibres of dimension $\leq 1$ and $\Ccal$ regular.  Let $\Ccal_0 \hookrightarrow \Ccal$ be the subscheme $f^{-1}(x)$.  Then the canonical homomorphism
\[
    \H^2(\Ccal,\G_m) \to \H^2(\Ccal_0,\G_m)
\]
induced by the closed immersion $\Ccal_0 \hookrightarrow \Ccal$ is bijective.
\end{theorem}
\begin{proof}
Note that for $\Ccal$ and $X$, $\Br$, $\Br'$ and $\H^2(-,\G_m)$ are equal since there is an ample sheaf (\cref{thm:BrauerGabber}) by~\cite{Liu2006}, p.~171, Corollary~1.36 and by regularity (\cref{thm:BrStrichGleichH2}).

Recall the definition of $X_n$ and $\Ccal_n$ from~\cref{lemma:BrauerProjlim}.  There are exact sequences of sheaves on $\Ccal_0$ for every $n \geq 0$ with the closed immersion $i_n: \Ccal_0 \hookrightarrow \Ccal_n$
\begin{equation} \label{eq:CoherentSheafAndGm}
    0 \to \Fcal \to i_{n+1}^*\G_{m,\Ccal_{n+1}} \to i_n^*\G_{m,\Ccal_n} \to 1
\end{equation}
with $\Fcal$ a coherent sheaf on $\Ccal_0$:  Zariski-locally on the source, $\Ccal \to X$ is of the form $\Spec(B) \to \Spec(A)$ and hence $\Ccal_n \to X_n$ of the form $\Spec(B/\mathfrak{m}^{n+1}) \to \Spec(A/\mathfrak{m}^{n+1})$.  There is an exact sequence
\[
    1 \to (1+\mathfrak{m}^n/\mathfrak{m}^{n+1}) \to (B/\mathfrak{m}^{n+1})^\times \to (B/\mathfrak{m}^n)^\times \to 1.
\]
By~\cref{lemma:deformationofunits}, the latter map is surjective since $\mathfrak{m}^n/\mathfrak{m}^{n+1} \subset B/\mathfrak{m}^{n+1}$ is nilpotent.  Now $(1+\mathfrak{m}^n/\mathfrak{m}^{n+1}) \isoto \mathfrak{m}^n/\mathfrak{m}^{n+1}$ is a coherent sheaf on $\Spec(B)$, the isomorphism given by $1 + x + \mathfrak{m}^2 \mapsto x + \mathfrak{m}^2$ for $x \in \mathfrak{m}$ (this is a homomorphism since $(1+x)(1+y) + \mathfrak{m}^2 = 1 + (x+y) + \mathfrak{m}^2$ for $x,y \in \mathfrak{m}$).  The sequences for a Zariski-covering of $\Ccal_0$ glue to an exact sequence of sheaves on $\Ccal_0$~\eqref{eq:CoherentSheafAndGm}.  

Therefore, the long exact sequence associated to~\eqref{eq:CoherentSheafAndGm} yields
\[
    \H^2(\Ccal_0, \Fcal) \to \H^2(\Ccal_0,i_{n+1}^*\G_{m,\Ccal_{n+1}}) \to \H^2(\Ccal_0,i_n^*\G_{m,\Ccal_n}) \to \H^3(\Ccal_0, \Fcal).
\]
Since $\Fcal$ is coherent, $\H^p_{\et}(\Ccal_0, \Fcal) = \H^p_{\Zar}(\Ccal_0, \Fcal)$ by~\cite{SGA42}, VII~4.3.  Thus, since $\dim{\Ccal_0} \leq 1$, $\H^2(\Ccal_0, \Fcal) = \H^3(\Ccal_0, \Fcal) = 0$.  Thus we get an isomorphism
\[
    \H^2(\Ccal_0,i_{n+1}^*\G_{m,\Ccal_{n+1}}) \isoto \H^2(\Ccal_0,i_n^*\G_{m,\Ccal_n})
\]
Next note that $\Ccal_0 \hookrightarrow \Ccal_n$ is a closed immersion defined by a nilpotent ideal sheaf, so by~\cref{lemma:cohomologyofnilimmersions} we get
\[
    \H^2(\Ccal_{n+1},\G_m) \isoto \H^2(\Ccal_n,\G_m).
\]
Taking torsion, it follows that $\Br'(\Ccal_{n+1}) \isoto \Br'(\Ccal_n)$, and then~\cref{thm:BrauerGabber} yields that the $\Br(\Ccal_{n+1}) \to \Br(\Ccal_n)$ are isomorphisms (in fact, injectivity suffices for the following).  Therefore the injectivity of $\Br(\Ccal) \to \Br(\Ccal_0)$ follows from~\cref{lemma:BrauerProjlim}.  One can apply this in our situation since the transition maps $\Pic(\Ccal_{n+1}) \to \Pic(\Ccal_n)$ are surjective by~\cite{EGAIV4}, p.~288, Corollaire~(21.9.12).

The surjectivity in~\cref{thm:H2bijective} is shown in a similar way:  Take an element of $\Br(\Ccal_0)$, represented by an Azumaya algebra $\underline{A}_0$.  As $\Br(\Ccal_{n}) \isoto \Br(\Ccal_0)$, see above, there is a compatible system of Azumaya algebras $\underline{A}_n$ on $\Ccal_n$.  Therefore, as above, there is an Azumaya algebra $\underline{\hat{A}}$ on $\hat{\Ccal}$.  Choose a locally free module $\underline{\Ecal}$ on $\Ccal$ such that there is an epimorphism $\underline{\hat{\Ecal}} \twoheadrightarrow \underline{\hat{A}}$ and consider the functor $F: (\mathrm{Sch}/X)^\circ \to (\mathrm{Set})$ defined by $F(X') = $ the set of pairs $(\underline{B}',p')$ where $\underline{B}'$ is a locally free module of $\underline{\Ecal} \otimes_X X'$ and $p'$ a multiplication law on $\underline{B}'$ which makes it into an Azumaya algebra.  Then $F$ is representable by a scheme locally of finite type over $X$ \emph{(loc.\ cit.)}, and the point in $F(x)$ (recall that $x$ is the closed point of $X$) defined by $\underline{A}_0$ gives us a point in $F(\hat{X})$, so by Artin approximation~\cite{ArtinApproximation}, p.~26, Theorem~(1.10) resp.\ Theorem~(1.12) it comes from a point in $F(X)$, which proves surjectivity.
\end{proof}

\begin{lemma} \label[lemma]{lemma:cdn}
Let $X$ be a scheme, $\ell$ a prime number and $n$ an integer such that $\cd_\ell(X) \leq n$.  Then $\H^i(X,\G_m)[\ell^\infty] = 0$ if $i > n+1$, resp.\ $i > n$ if $\ell$ is invertible on $X$.
\end{lemma}
\begin{proof}
If $\ell$ is invertible on $X$, the Kummer sequence
\[
    1 \to \mu_{\ell^r} \to \G_m \stackrel{\ell^r}{\to} \G_m \to 1
\]
induces a long exact sequence in cohomology, part of which is
\[
    0 = \H^{i+1}(X,\mu_{\ell^r}) \to \H^{i+1}(X, \G_m) \stackrel{\ell^r}{\to} \H^{i+1}(X, \G_m) \to \H^{i+2}(X,\mu_{\ell^r}) = 0,
\]
for $i > n$, i.\,e.\ multiplication by $\ell^r$ induces an isomorphism on $\H^{i+1}(X, \G_m)$ for $i > n$, so the claim $\H^i(X,\G_m)[\ell^\infty] = 0$ for $i > n$ follows.

For general $\ell$, one has an exact sequence
\[
    1 \to \ker(\ell^r) \to \G_m \stackrel{\ell^r}{\to} \G_m \to \coker(\ell^r) \to 1
\]
of étale sheaves which splits up into
\[\xymatrix{
    1 \ar[r] &\ker(\ell^r) \ar[r] &\G_m \ar[rr]^{\ell^r} \ar@{->>}[rd] && \G_m \ar[r] &\coker(\ell^r) \ar[r] &1\\
             &                    &                                    &  \im(\ell^r) \ar@{^{(}->}[ru]   }
\]
By the same argument as in the case $\ell$ invertible on $X$, one finds $\H^i(X,\G_m) \isoto \H^i(X,\im(\ell^r))$ for $i > n$, and, since $\ell^r\coker(\ell^r) = 0$, $\H^i(X,\im(\ell^r)) \isoto \H^i(X,\G_m)$ for $i > n+1$.  So, altogether
\[
    \ell^r: \H^i(X,\G_m) \isoto \H^i(X,\im(\ell^r)) \isoto \H^i(X,\G_m)
\]
is injective for $i > n+1$, and therefore $\H^i(X,\G_m)[\ell^\infty] = 0$ for $i > n+1$.
\end{proof}

\begin{lemma} \label[lemma]{lemma:BrauerOfProperCurve}
Let $C/K$ be a projective regular curve over a separably closed field.  Then $\Br(C) = \Br'(C) = \H^2(C,\G_m) = 0$.
\end{lemma}
\begin{proof}
One has $\Br(C) = \Br'(C) = 0$ by~\cite{GroupeDeBrauerIII}, p.~132, Corollaire~(5.8) since $C$ is a proper curve over a separably closed field.  Moreover, \cref{thm:BrStrichGleichH2} implies $\Br'(C) = \H^2(C,\G_m)$.
\end{proof}

\begin{theorem} \label{thm:RqGm=0q>1}
Let $\pi: \Ccal \to X$ be projective and smooth with $\Ccal$ and $X$ regular, all fibres of dimension $1$ and $X$ a variety. Then
\[
    \R^q\pi_*\G_m = 0 \quad\text{for $q > 1$.}
\]
\end{theorem}
\begin{proof}
By~\cite{SGA42}, VIII 5.2, resp.~\cite{MilneÉtaleCohomology}, p.~88, III.1.15 one can assume $X$ strictly local and we must prove $\H^i(\Ccal, \G_m) = 0$ for $i > 1$.  By the proper base change theorem~\cite{MilneÉtaleCohomology}, p.~224, Corollary~VI.2.7, for torsion sheaves $\Fcal$ on $\Ccal$ with restriction $\Fcal_0$ to the closed fibre $\Ccal_0$, one has restriction isomorphisms
\[
    \H^i(\Ccal,\Fcal) \to \H^i(\Ccal_0, \Fcal_0).
\]
Since $\dim \Ccal_0 = 1$, the latter term vanishes for $i > 2$, and for $i > 1$ if $\Fcal$ is $p$-torsion, where $p$ is the residue field characteristic.  Therefore
\[
    \cd(\Ccal) \leq 2, \text{ and } \cd_p(\Ccal) \leq 1.
\]
The relation $\H^i(\Ccal,\G_m) = 0$ for $i > 2$ follows from the fact that these groups are torsion by~\cref{thm:LichtenbaumBrauer}\,(a) and from~\cref{lemma:cdn}.

It remains to treat the case $i = 2$, i.\,e.\ to prove
\[
    \H^2(\Ccal,\G_m) = 0.
\]
If $\ell$ is invertible on $X$, then
\[
    \H^2(\Ccal,\G_m)[\ell^\infty] = 0
\]
follows as in the case $i > 2$.  From the Kummer sequence~\cref{thm:KurzeExakteSequenzAusKummerSequenz} and the inclusion $\Ccal_0 \hookrightarrow \Ccal$, one gets a commutative diagram with exact rows
\[\xymatrix{
0       \ar[r]  & \Pic(\Ccal)\otimes\Q_\ell/\Z_\ell   \ar[r] \ar@{->>}[d] &\H^2(\Ccal,\mu_{\ell^\infty})   \ar[r] \ar[d]^\iso &\H^2(\Ccal,\G_m)[\ell^\infty] \ar[d]  \ar[r] &0\\
0       \ar[r]  & \Pic(\Ccal_0)\otimes\Q_\ell/\Z_\ell \ar[r] &\H^2(\Ccal_0,\mu_{\ell^\infty}) \ar[r] &\H^2(\Ccal_0,\G_m)[\ell^\infty] \ar[r] &0,}\]
where the middle vertical arrow is bijective by proper base change~\cite{MilneÉtaleCohomology}, p.~224, Corollary~VI.2.7, and the first vertical arrow is surjective by~\cite{EGAIV4}, p.~288, Corollaire~(21.9.12) and the right exactness of the tensor product.  Hence, by the five lemma, the right vertical morphism ist bijective, and $\H^2(\Ccal,\G_m)[\ell^\infty] = 0$ since $\H^2(\Ccal_0,\G_m)[\ell^\infty] = 0$ by~\cref{lemma:BrauerOfProperCurve}.

For general $\ell$, one uses~\cref{thm:H2bijective}, which gives us
\[
    \H^2(\Ccal,\G_m) \isoto \H^2(\Ccal_0,\G_m),
\]
and $\H^2(\Ccal_0,\G_m) = 0$ by~\cref{lemma:BrauerOfProperCurve}.
\end{proof}

\begin{remark}
Note that the difficult~\cref{thm:H2bijective} is only needed for the $p$-torsion in~\cref{thm:RqGm=0q>1}.
\end{remark}

\begin{corollary} \label[corollary]{cor:LangeExakteSequenzMitGm}
In the situation of~\cref{thm:RqGm=0q>1}, assume we have locally Noetherian separated schemes with geometrically reduced and connected fibres.  Then one has the long exact sequence
\begin{align*}
	0 \to &\H^1(X, \G_m) \stackrel{\pi^*}{\to} \H^1(\Ccal, \G_m) \to \H^0(X, \R^1\pi_*\G_m) \to \ldots\\
	&\H^q(X, \G_m) \stackrel{\pi^*}{\to} \H^q(\Ccal, \G_m) \to \H^{q-1}(X, \R^1\pi_*\G_m) \to \ldots
\end{align*}
\end{corollary}

\begin{lemma} \label[lemma]{lemma:finiteMorphismOfRelativeCurves}
Let $\pi: \Ccal \to X$ be a proper relative curve with $X$ integral and let $D$ be an irreducible Weil divisor in the generic fibre $C$ of $\pi$.  Let $\bar{D}$ be the closure of $D$ in $\Ccal$.  Then $\pi_{\bar{D}}: \bar{D} \to X$ is a finite morphism.
\end{lemma}
\begin{proof}
Take $D$ an irreducible Weil divisor in $C$, i.\,e. a closed point of $C$.  Then $\bar{D}/X$ is finite of degree $\deg(D)$:  $\bar{D}/X$ is of finite type, generically finite and dominant (since $D$ lies over the generic point of $X$) and $\bar{D}$ (as a closure of an irreducible set) and $X$ are irreducible, $\Ccal$ and $X$ are integral, hence by~\cite{HartshorneAG}, p.~91, Exercise~II.3.7 there is a dense open subset $U \subseteq X$ such that $\bar{D}|_U \to U$ is finite.  Since $\pi|_{\bar{D}}: \bar{D} \to X$ is proper (since it factors as a composition of a closed immersion and a proper morphism $\bar{D} \hookrightarrow \Ccal \stackrel{\pi}{\to} X$), it suffices to show that it is quasi-finite.  If there is an $x \in X$ such that $\pi|_{\bar{D}}^{-1}(x)$ is not finite, we must have $\pi|_{\bar{D}}^{-1}(x) = \Ccal_x$, i.\,e.\ the whole curve as a fibre, since the fibres of $\pi$ are irreducible.  But then $\bar{D}$ would have more than one irreducible component, a contradiction.
\end{proof}

\begin{corollary} \label[corollary]{cor:BrauerUndShaExakteSequenz}
If, in the situation of~\cref{cor:LangeExakteSequenzMitGm}, $\pi$ has a section $s: X \to \Ccal$, one has split short exact sequences
\[\begin{tikzpicture}[->]
        \matrix(m)[matrix of nodes,column sep=0.5cm]{$0$&$\H^i(X, \G_m)$&$\H^i(\Ccal, \G_m)$&$\H^{i-1}(X, \R^1\pi_*\G_m)$&$0$\\};
        \draw(m-1-1)--(m-1-2);
        \draw(m-1-2)--node[below]{$\pi^*$}(m-1-3);
        \draw(m-1-3)to[in=20,out=160]node[above]{$s^*$}(m-1-2);
        \draw(m-1-3)--(m-1-4);
        \draw(m-1-4)--(m-1-5);
\end{tikzpicture}\]
for $i \geq 1$.

In the general case, denote by $C/K$ the generic fibre of $\Ccal/X$, and assume that for every Weil divisor $D$ in $C$, $\bar{D} \subseteq \Ccal$ has everywhere the same dimension as $X$ with no embedded components, and denote by $\delta$ the greatest common divisor of the degrees of Weil divisors on $C/K$, i.\,e.\ the index of $C/K$.  Then one has an exact sequence
\[
    0 \to K_2 \to \H^2(X, \G_m) \stackrel{\pi^*}{\to} \H^2(\Ccal, \G_m) \to \H^1(X, \R^1\pi_*\G_m) \to K_3 \to 0,
\]
where $K_i = \ker(\H^i(X, \G_m) \stackrel{\pi^*}{\to} \H^i(\Ccal, \G_m))$ are Abelian groups annihilated by $\delta$ whose prime-to-$p$ torsion is finite.
\end{corollary}
\begin{proof}
The first assertion is obvious from the previous~\cref{cor:LangeExakteSequenzMitGm} and the existence of a section.

For the second claim, take an irreducible Weil divisor $D$ in $C$.  Then $\pi_{\bar{D}}: \bar{D} \to X$ is a finite morphism by~\cref{lemma:finiteMorphismOfRelativeCurves}.  We have the commutative diagram
\[\xymatrix{
\bar{D} \ar@{^{(}->}[r]^i \ar[rd]_{\pi|_{\bar{D}}} &\Ccal \ar[d]^\pi\\
                                &X.}\]
By the Leray spectral sequence, we have that $\H^i(\bar{D},\G_m) = \H^i(X,\pi|_{\bar{D},*}\G_m)$ as $\pi|_{\bar{D}}$ is finite, hence exact for the étale topology, see~\cite{MilneÉtaleCohomology}, p.~72, Corollary~II.3.6.  If $\pi|_{\bar{D}}$ is also flat, by finite locally freeness we have a norm map $\pi|_{\bar{D},*}\G_m \to \G_m$ whose composite with the inclusion $\G_m \to \pi|_{\bar{D},*}\G_m$ is the $\delta$-th power map.  If not, there is still a norm map since $f$ is flat in codimension $1$ since $\bar{D}$ has everywhere the same dimension as $X$ with no embedded components, so one can take the norm there, which then will land in $\G_m$ as $X$ is normal.

We have ${\pi|_{\bar{D}}}_* \circ i^* \circ \pi^* = {\pi|_{\bar{D}}}_* \circ \pi|_{\bar{D}}^* = \deg(D)$, so $\ker(\pi^*: \H^*(X,\G_m) \to \H^*(\Ccal,\G_m))$ is annihilated by $\deg(D)$ for all $D$, hence by the index $\delta$.  Now the finiteness of the prime-to-$p$ part of the $K_i$ follows from~\cref{thm:LichtenbaumBrauer}.
\end{proof}

In the following, assume $\pi$ is smooth (automatically projective since $\Ccal, X$ are projective varieties), with all geometric fibres integral and of dimension $1$, and that it has a section $s: X \to \Ccal$. Assume further that $\pi_*\Ocal_\Ccal = \Ocal_X$ holds universally and $\pi$ is cohomologically flat in dimension $0$.  This holds, e.\,g., if $\pi$ is a flat proper morphism of locally Noetherian separated schemes with geometrically connected fibres.

We recall some definitions from~\cite{FGAExplained}, p.~252, Definition~9.2.2.
\begin{definition}
The \defn{relative Picard functor} $\Pic_{X/S}$ on the category of (locally Noetherian) $S$-schemes is defined by $\Pic_{X/S}(T) := \Pic(X \times_S T)/\pr_2^*\Pic(T)$.  Its associated sheaves in the Zariski, étale and fppf topology are denoted by $\Pic_{X/S,\Zar}$, $\Pic_{X/S,\et}$ and $\Pic_{X/S,\fppf}$.
\end{definition}

Now we come to the representability of the relative Picard functor by a group scheme, whose connected component of unity is an Abelian scheme.
\begin{theorem} \label{thm:RepresentabilityOfPic}
$\Pic_{\Ccal/X}$ is represented by a separated smooth $X$-scheme $\PPic_{\Ccal/X}$ locally of finite type. $\Pic^0_{\Ccal/X}$ is represented by an Abelian $X$-scheme $\PPic^0_{\Ccal/X}$.  For every $T/X$,
\[
    0 \to \Pic(T) \to \Pic(\Ccal \times_X T) \to \PPic_{\Ccal/X}(T) \to 0
\]
is exact.
\end{theorem}
\begin{proof}
Since $\pi$ has a section $s: X \to \Ccal$ and $\pi_*\Ocal_\Ccal = \Ocal_X$ holds universally, by~\cite{FGAExplained}, p.~253, Theorem~9.2.5
\[
    \Pic_{\Ccal/X} \isoto \Pic_{\Ccal/X, \Zar} \isoto \Pic_{\Ccal/X, \et} \isoto \Pic_{\Ccal/X, \fppf}.
\]
Since $\pi$ is projective and flat with geometrically integral fibres, by~\cite{FGAExplained} p.~263, Theorem~9.4.8, $\PPic_{\Ccal/X}$ exists, is separated and locally of finite type over $X$ and represents $\Pic_{\Ccal/X, \et}$.  Since $X$ is Noetherian and $\Ccal/X$ projective, by~\emph{loc.\ cit.}, $\PPic_{\Ccal/X}$ is a disjoint union of open subschemes, each an increasing union of open quasi-projective $X$-schemes.  By~\cite{BLR}, p.~259\,f., Proposition~4, $\PPic^0_{\Ccal/X}/X$ is an Abelian scheme (this uses that $\Ccal/X$ is a relative \emph{curve} or an Abelian scheme).  The last assertion of~\cref{thm:RepresentabilityOfPic} follows from~\cite{BLR}, p.~204, Proposition~4.
\end{proof}

\section{The Néron model}
\begin{lemma} \label[lemma]{lemma:HenselisationOfRegularLocalRing}
Let $A$ be a regular local ring.  Then $A^h \otimes_A \Quot(A) = \Quot(A^h)$ and $A^{sh} \otimes_A \Quot(A) = \Quot(A^{sh})$.
\end{lemma}
\begin{proof}
By~\cite{MilneÉtaleCohomology}, p.~38, Remark~4.11, $A^{sh}$ is the localisation at a maximal ideal lying over the maximal ideal $\mathfrak{m}$ of $A$ of the integral closure of $A$ in $(\Quot(A)^{\mathrm{sep}})^I$ with $I \subseteq \Gal(\Quot(A)^{\mathrm{sep}}/\Quot(A))$ the inertia subgroup.  Pick an element $a \in A^{sh}$.  It is a root of a monic polynomial $f(T) \in A[T]$, which we can assume to be monic irreducible since $A$ is a regular local ring and hence factorial by~\cite{Matsumura}, p~163, Theorem~20.3, and hence $A[T]$ is factorial by~\cite{Matsumura}, p.~168, Exercise~20.2.  Since $A$ is factorial and $f \in A[T]$ is irreducible, by the lemma of Gauß~\cite{Bosch}, p.~64, Korollar~6, $f$ is also irreducible over $\Quot(A)$.  Hence $A[a] \otimes_A \Quot(A) = \Quot(A)[T]/(f(T))$ is a field, and $A^{sh} \otimes_A \Quot(A)$ is a directed colimit of fields since tensor products commute with colimits, and hence itself a field, namely $\Quot(A^{sh})$. Now note that localisation does not change the quotient field.

The proof for $A^h$ is the same:  Just replace the inertia group by the decomposition group.
\end{proof}

\begin{lemma} \label[lemma]{lemma:Euler characteristic locally constant}
Let $S$ be a locally Noetherian scheme, $f: X \to S$ be a proper flat morphism and $\Ecal$ a locally free sheaf on $X$.  Then the Euler characteristic
\[
    \chi_\Ecal(s) = \sum_{i \geq 0} (-1)^i \dim \H^i(X_s, \Ecal_s)
\]
is locally constant on $S$.
\end{lemma}
\begin{proof}
See~\cite{EGAIII2}, p.~76\,f., Théorème (7.9.4).
\end{proof}

\begin{theorem}[The Néron model] \label{thm:PicIsNeronModel}
Let $S$ be a regular, Noetherian, integral, separated scheme, and $g: \{\eta\} \hookrightarrow S$ the inclusion of the generic point.  Let $X/S$ be a smooth projective variety with geometrically integral fibres that admits a section such that its Picard functor is representable (e.\,g., $X/S$ a smooth projective curve admitting a section or an Abelian scheme).  Then $\PPic_{X/S} \isoto g_*g^*\PPic_{X/S}$ as sheaves on $S_{\mathrm{sm}}$, the smooth site.  Let $\Acal/S$ be an Abelian scheme.  Then
\[
    \Acal \isoto g_*g^*\Acal
\]
as sheaves on $S_{\mathrm{sm}}$.
\end{theorem}
The theorem says that an Abelian scheme and a Picard scheme is the Néron model of its generic fibre.

\begin{proof}
The main idea for injectivity is to use the separatedness of our schemes, and the main idea for surjectivity is that Weil divisors spread out and that the Picard group equals the Weil divisor class group by regularity.

We first prove everything for sheaves on the \emph{étale} site on $S$.

Let $f: \PPic_{X/S} \to g_*g^*\PPic_{X/S}$ be the natural map of étale sheaves induced by adjointness.  Let $\bar{s} \to S$ be a geometric point.  We have to show that $(\coker(f))_{\bar{s}} = 0 = (\ker(f))_{\bar{s}}$.

Taking stalks and using~\cite{EGAIV3}, p.~52, Proposition~8.14.2, we get the following commutative diagram:
\[\xymatrix{
\PPic_{X/S}(\Ocal_{S,s}^{sh})        \ar[r]^{f}  &\PPic_{X/S}(\Quot(\Ocal_{S,s}^{sh}))   \ar@{->>}[r] &(\coker(f))_{\bar{s}}\\
\Pic(X \times_S \Ocal_{S,s}^{sh})   \ar@{->>}[u]\ar[r]    &\Pic(X \times_S \Quot(\Ocal_{S,s}^{sh}))\ar@{->>}[u]\\
\Div(X \times_S \Ocal_{S,s}^{sh})   \ar@{->>}[u]\ar[r]    &\Div(X \times_S \Quot(\Ocal_{S,s}^{sh}))\ar@{->>}[u]}\]
Note that $\Ocal_{S,s}^{sh}$ is a domain since it is regular as a strict Henselisation of a regular local ring by~\cite{LeiFuÉtaleCohomologyTheory}, p.~111, Proposition~2.8.18.  By~\cite{EGAIV3}, p.~52, Proposition~8.14.2, one has $(\PPic_{X/S})_{\bar{s}} = \PPic_{X/S}(\Ocal_{S,s}^{sh})$ and $(g_*g^*\PPic_{X/S})_{\bar{s}} = \PPic_{X/S}(\Ocal_{S,s}^{sh} \otimes_{\Ocal_{S,s}} K(S))$. But for a regular local ring $A$, one has $A^{sh} \otimes_A \Quot(A) = \Quot(A^{sh})$ by~\cref{lemma:HenselisationOfRegularLocalRing}.

By~\cite{HartshorneAG}, p.~145, Corollary~II.6.16, one has a surjection from $\Div$ to $\Pic$ since $S$ is Noetherian, integral, separated and locally factorial.  By~\cref{thm:RepresentabilityOfPic}, the upper vertical arrows are surjective (under the assumption that $X/S$ has a section).

But here, the lower horizontal map is surjective:  A preimage under $\iota: X \times_S \Quot(\Ocal_{S,s}^{sh}) \to X \times_S \Ocal_{S,s}^{sh}$ of $D \in \Div(X \times_S \Quot(\Ocal_{S,s}^{sh}))$ is $\bar{D} \in \Div(X \times_S \Ocal_{S,s}^{sh})$, the closure taken in $X \times_S \Ocal_{S,s}^{sh}$.  In fact, note that $D$ is closed in $X \times_S \Quot(\Ocal_{S,s}^{sh})$ since it is a divisor; the closure of an irreducible subset is irreducible again, and the codimension is also $1$ since the codimension is the dimension of the local ring at the generic point $\eta_D$ of $D$, and the local ring of $\eta_D$ in $X \times_S \Quot(\Ocal_{S,s}^{sh})$ is the same as the local ring of $\eta_D$ in $X \times_S \Ocal_{S,s}^{sh}$ as it is the colimit of the global sections taken for all open neighbourhoods of $\eta_D$.  Hence $(\coker(f))_{\bar{s}} = 0$. 

For $(\ker(f))_{\bar{s}} = 0$, consider the diagram
\[\xymatrix{
                    &\Spec{\Quot(\Ocal_{S,s}^{sh})} \ar@{^{(}->}[ld] \ar@{^{(}->}[dd] \ar[r] &\PPic_{X/S}\ar[dd]\\
\Spec\Ocal \ar[rd] \ar@{-->}[rru] \\
                        &\Spec{\Ocal_{S,s}^{sh}} \ar@{-->}[ruu] \ar[r] &S.}\]
We want to show that a lift $\Spec{\Ocal_{S,s}^{sh}} \to \PPic_{X/S}$ of $\Spec{\Quot(\Ocal_{S,s}^{sh})} \to \PPic_{X/S}$ is unique. As $\PPic_{X/S}/S$ is separated, this is true for all valuation rings $\Ocal \subset \Quot(\Ocal_{S,s}^{sh})$ by the valuative criterion of separatedness~\cite{EGAII}, p.~142, Proposition~(7.2.3).  But by~\cite{Matsumura}, p.~72, Theorem~10.2, every local ring $(\Ocal_{S,s}^{sh})_\p$ is dominated by a valuation ring $\Ocal$ of $\Quot(\Ocal_{S,s}^{sh})$.  It follows from the valuative criterion for separatedness that the lift is topologically unique. Assume $\phi,\phi'$ are two lifts. Now cover $\PPic_{X/S}/S$ by open affines $U_i = \Spec{A_i}$ and their preimages $\phi^{-1}(U_i) = \phi'^{-1}(U_i)$ by standard open affines $\{D(f_{ij})\}_j$.
\[\xymatrix{
                    & &A_i \ar[lld] \ar[ldd]^{\phi^\#}\ar@<1ex>[ldd]_{\phi'^\#}\\
\Ocal_{f_{ij}} \\
                        &(\Ocal_{S,s}^{sh})_{f_{ij}}  \ar@{^{(}->}[lu]  }\]
It follows that $\phi = \phi'$.


Now we prove the last statement of the theorem for Abelian schemes, so one can deduce the statement for Abelian varieties by noting that $\Acal = (\Acal^\vee)^\vee = \PPic^0_{\Acal^\vee/S}$.

We want to show that
\begin{align} \label{eq:PPic0}
    \PPic_{X/S}^0(S') \to \PPic_{X/S}^0(S'_{\eta})
\end{align}
is bijective for any étale $S$-scheme $S'$.

First note that such an $S'$ is regular, so its connected components are integral, and $S'_{\eta}$ is the disjoint union of the generic points of the connected components of $S'$, so we can replace $X \to S$ with the restrictions of the base change $X' \to S'$ over each connected component of $S'$ separately to reduce to checking for $S' = S$.

For any section $S \to \PPic_{X/S}$, since $S$ is connected and $\PPic_{X/S}^{\tau}$ is open and closed in $\PPic_{X/S}$ by~\cite{SGA6}, p.~647\,f., exp.~XIII, Théorème~4.7, the preimage of $\PPic_{X/S}^{\tau}$ under the section is open and closed in $S$, hence empty or $S$.  Thus, if even a single point of $S$ is carried into $\PPic_{X/S}^{\tau}$ under the section, then the whole of $S$ is.  More generally, when using $S'$-valued points of $\PPic_{X/S}$ for any étale $S$-scheme $S'$, such a point lands in $\PPic_{X/S}^{\tau}$ if and only if some point in each connected component of $S'$ does, such as the generic point of each connected component $S'_{\eta}$.

This proves the statement in~\eqref{eq:PPic0} for $\PPic_{X/S}^0$ replaced by $\PPic_{X/S}^{\tau}$, and hence for $\PPic_{X/S}^0$ in cases where it coincides with $\PPic_{X/S}^{\tau}$ (i.\,e., when the geometric fibers have component group for $\PPic_{X/S}$---i.\,e., Néron-Severi group---that is torsion-free, e.\,g.\ for Abelian schemes or curves).

The Néron mapping property for \emph{smooth} $S$-schemes $S'$ follows formally.  Since $S' \to S$ is smooth, $S'$ is regular as well and its generic points lie over $\eta$.  Since an $S'$-point of $\PPic_{X/S}$ is the same as an $S'$-point of $\PPic_{X'/S'}$, an $S'_\eta$-point extends uniquely to an $S'$-point.
\end{proof}

\section{The Tate-Shafarevich group}

\subsection{Comparison with the classical definition}
The main theorem of this subsection is~\cref{thm:codim1Punktereichen}, which shows that our definition of the Tate-Shafarevich group is analogous to the classical definition.  \cref{lemma:HenselisationAndCompletion} gives an even more direct connection to the classical case, involving completions instead of Henselisations.

\begin{proposition}
Let $X$ be integral and $\Acal/X$ be an Abelian scheme over $X$ regular, Noetherian and separated.  Then $\H^i(X,\Acal)$ is torsion for $i > 0$.
\end{proposition}
\begin{proof}
Consider the Leray spectral sequence for the inclusion $g: \{\eta\} \hookrightarrow X$ of the generic point
\[
    \H^p(X, \R^qg_*g^*\Acal) \Rightarrow \H^{p+q}(\eta, g^*\Acal).
\]
Calculation modulo the Serre subcategory of torsion sheaves, and exploiting the fact that Galois cohomology groups are torsion in dimension $> 0$, and therefore also the higher direct images $\R^qg_*g^*\Acal$ are torsion sheaves by~\cite{MilneÉtaleCohomology}, p.~88, Proposition~III.1.13, the spectral sequence degenerates modulo the Serre subcategory of torsion sheaves giving
\[
    \H^p(X, g_*g^*\Acal) = E^{p,0}_2 = E^p = \H^p(\eta, g^*\Acal) = 0
\]
for $p > 0$ modulo torsion.  Because of the Néron mapping property we have $\H^p(X, \Acal) \isoto \H^p(X, g_*g^*\Acal)$, which finishes the proof.
\end{proof}

\begin{definition} \label[definition]{def:Sha}
Define the \defn{Tate-Shafarevich group} of an Abelian scheme $\Acal/X$ by
\[
    \Sha(\Acal/X) = \Het^1(X,\Acal).
\]
\end{definition}

\begin{remark}
Note that if $A$ is an Abelian variety over a global field $K$, $\Sha(A/K)$ is \emph{not} the usual Tate-Shafarevich group of $A$ over $K$.  One rather has to choose a Néron model $\Acal$ over a model $X$ of $K$.
\end{remark}

\begin{theorem} \label[theorem]{thm:ShaSequenzAllePunkte}
Let $X$ be regular, Noetherian, integral and separated and let $\Acal$ be an Abelian scheme over $X$.  For $x \in X$, denote the function field of $X$ by $K$, the quotient field of the strict Henselisation of $\Ocal_{X,x}$ by $K_x^{nr}$, the inclusion of the generic point by $j: \{\eta\} \hookrightarrow X$ and let $j_x: \Spec(K_x^{nr}) \hookrightarrow \Spec(\Ocal_{X,x}^{sh}) \hookrightarrow X$ be the composition. Then we have
\[
	\H^1(X, \Acal) \isoto \ker\Big(\H^1(K, j^*\Acal) \to \prod_{x \in X} \H^1(K_x^{nr}, j_x^*\Acal) \Big).
\]
\end{theorem}
\begin{proof}
The Leray spectral sequence $\H^p(X, \R^qj_*(j^*\Acal))) \Rightarrow \H^{p+q}(K, j^*\Acal)$ yields the exactness of $0 \to E_2^{1,0} \to E^1 \to E_2^{0,1}$, i.\,e.\
\[
	\H^1(X, j_*j^*\Acal) = \ker\left(\H^1(K, j^*\Acal) \to \H^0(X, \R^1j_*(j^*\Acal)) \right).
\]
Since
\[
	\H^0(X, \R^1j_*(j^*\Acal)) \to \prod_{x \in X}\R^1j_*(j^*\Acal)_{\bar{x}}
\]
is injective (\cite{MilneÉtaleCohomology}, p.~60, Proposition~II.2.10:  If a section of an étale sheaf is non-zero, there is a geometric point for which the stalk of the section is non-zero) and
\[
    \R^1j_*(j^*\Acal)_{\bar{x}} = \H^1(K_x^{nr}, j_x^*\Acal)
\]
by~\cref{lemma:HenselisationOfRegularLocalRing}, the theorem follows from the kernel-cokernel exact sequence and the Néron mapping property $\H^1(X, \Acal) \isoto \H^1(X, j_*j^*\Acal)$.
\end{proof}

By abuse of notation, we also denote $\Spec(\Ocal_{X,x}^{h}) \hookrightarrow X$ and $\Spec(K_x^{h}) \hookrightarrow \Spec(\Ocal_{X,x}^{h}) \hookrightarrow X$ by $j_x$.

\begin{theorem} \label[theorem]{thm:codim1Punktereichen}
In the situation of~\cref{thm:ShaSequenzAllePunkte}, one can replace the product over all points by the following:

(a) the closed points:  One has isomorphisms
\begin{align} \label{eq:abgPunkte}
	\H^1(X, \Acal) \isoto \ker\Big(\H^1(K, j^*\Acal) \to \prod_{x \in |X|} \H^1(K_x^{nr}, j_x^*\Acal) \Big)
\end{align}
and
\begin{align} \label{eq:abgPunkteendlicherGrundkoerper}
    \H^1(X, \Acal) \isoto \ker\Big(\H^1(K, j^*\Acal) \to \prod_{x \in |X|} \H^1(K_x^{h}, j_x^*\Acal) \Big)
\end{align}
with $K_x^h = \Quot(\Ocal_{X,x}^h)$ if $\kappa(x)$ is finite.

or (b) the codimension-$1$ points:  One has an isomorphism
\begin{align}
	\H^1(X, \Acal) \isoto \ker\Big(\H^1(K, j^*\Acal) \to \bigoplus_{x \in X^{(1)}} \H^1(K_x^{nr}, j_x^*\Acal) \Big) \label{eq:Codim1}
\end{align}
if one disregards the $p$-torsion ($p = \Char{k}$), $X/k$ is smooth projective over $k$ finitely generated and $\Acal/X$ is an Abelian scheme such that the vanishing condition
\begin{equation} \label{eq:vanishingcondition}
    \H^i_Z(X, \Acal) = 0\quad\text{for $i \leq 2$}.
\end{equation}
for $Z \hookrightarrow X$ be a reduced closed subscheme of codimension $\geq 2$ is satisfied.  For $\dim{X} \leq 2$, this also holds for the $p$-torsion.

For $x \in X^{(1)}$, one can also replace $K_x^{nr}$ and $K_x^h$ by the quotient field of the completions $\hat{\Ocal}_{X,x}^{sh}$ and $\hat{\Ocal}_{X,x}^h$, respectively.
\end{theorem}
The theorem says that our~\cref{def:Sha} of the Tate-Shafarevich group is indeed a generalisation of the classical definition $\Sha(A/K) = \ker(\H^1(K,A) \to \oplus_{v\in M_K}\H^1(K_v,A))$ of the Tate-Shafarevich group of an Abelian variety $A$ over a global field $K$ with set of places $M_K$ and completions $K_v$ (where we have replaced the completions $K_v$ by our $K_x^{nr}$ or $\Quot(\hat{\Ocal}_{X,x}^{sh})$, respectively).  For the statement on completions see~\cref{lemma:HenselisationAndCompletion} below.

\begin{proof}[Proof of~\cref{thm:codim1Punktereichen}\,(a)] \label{ProofOfAbgPts}
For the proof of~\eqref{eq:abgPunkte}, note that, in the proof of~\cref{thm:ShaSequenzAllePunkte}, even
\[
	\H^0(X, \R^1j_*(j^*\Acal)) \to \prod_{x \in |X|}\R^1j_*(j^*\Acal)_{\bar{x}}
\]
is injective by~\cite{MilneÉtaleCohomology}, p.~65, Remark~II.2.17\,(b):  If a section of an étale sheaf is non-zero, there is a \emph{closed} geometric point for which the stalk of the section is non-zero.  This is because $X/k$ is a variety and hence Jacobson.

To deduce~\eqref{eq:abgPunkteendlicherGrundkoerper} from~\eqref{eq:abgPunkte}, note that in view of the inflation-restriction exact sequence
\[
    0 \to \H^1(K_x^{nr}/K_x^h,j_x^*\Acal) \to \H^1(K_x^h,j_x^*\Acal) \to \H^1(K_x^{nr},j_x^*\Acal)
\]
it suffices to show that $\H^1(K_x^{nr}/K_x^h,j_x^*\Acal) = 0$.  But one has in general for all $i \geq 0$ and étale sheaves $\Fcal$ on $\Spec(R)$, $R$ a Henselian local ring with closed point $s$, geometric point $\bar{s}$ and strict Henselisation $R^{sh}$:  The stalk is $\Fcal_{\bar{s}} = \Fcal(R^{sh})$ as a $\Gal_{\kappa(s)} = \Gal(R^{sh}/R)$-module, and $\H^i(R,\Fcal) = \H^i(\Gal(R^{sh}/R),\Fcal)$.  Thus one gets for $R = \Ocal_{X,x}^h$ and $\Fcal = \Acal$ using that $\Acal(\Ocal_{X,x}^{sh}) = \Acal(K_x^{nr})$ by the Néron mapping property~\cref{thm:PicIsNeronModel}
\[
    \H^1(K_x^{nr}/K_x^h,j_x^*\Acal) = \H^1(\Gal(\Ocal_{X,x}^{sh}/\Ocal_{X,x}^h),j_x^*\Acal) = \H^1(\Ocal_{X,x}^h,j_x^*\Acal) \stackrel{(1)}{=} \H^1(\kappa(x),j_x^*\Acal) \stackrel{(2)}{=} 0.
\]
Here, (1) holds by~\cite{MilneÉtaleCohomology}, p.~116, Remark~III.3.11\,(a), and (2) by the Lang-Steinberg theorem~\cite{MumfordAbelianVarieties}, p.~205, Theorem~3 using that $\kappa(s)$ is finite.
\end{proof}

After several lemmas, the proof of~\cref{thm:codim1Punktereichen}\,(b) will finally follow from~\cref{lemma:Surjectivity} and~\cref{lemma:Injectivity}.

The vanishing condition~\eqref{eq:vanishingcondition} for Picard schemes $\Acal = \PPic^0_{\Ccal/X}$ will be established in~\cref{lemma:VerschwindungfuerCodim2Pic0} below.

We first establish some vanishing results for étale cohomology with supports.

\begin{lemma} \label[lemma]{lemma:HiZGm=0}
Let $X/k$ be a smooth variety and $Z \hookrightarrow X$ be a reduced closed subscheme of codimension $\geq 2$.  Then one has
\[
    \H^i_Z(X, \G_m) = 0\quad\text{for $i \leq 2$},
\]
and for $i = 3$ at least away from $p$.
\end{lemma}
\begin{proof}
See~\cite{GroupeDeBrauerIII}, p.~133\,ff.:  Using the local-to-global spectral sequence (\cite{GroupeDeBrauerIII}, p.~133, (6.2))
\[
    E_2^{p,q} = \H^p(X, \Hcal_Z^q(\G_m)) \Rightarrow \H^{p+q}_Z(X,\G_m)
\]
and~\cite{GroupeDeBrauerIII}, p.~133--135
\begin{align*}
    \Hcal_Z^0(\G_m) &= 0\quad\text{\cite{GroupeDeBrauerIII}, p.~133, (6.3)}\\
    \Hcal_Z^1(\G_m) &= 0\quad\text{\cite{GroupeDeBrauerIII}, p.~133, (6.4) since the codimension of $Z$ in $X$ is $\neq 1$}\\
    \Hcal_Z^2(\G_m) &= 0\quad\text{\cite{GroupeDeBrauerIII}, p.~134, (6.5)}\\
    \Hcal_Z^3(\G_m)^{(p')} &= 0\quad\text{\cite{GroupeDeBrauerIII}, p.~134\,f., Thm.~(6.1)},
\end{align*}
(even $\Hcal_Z^3(\G_m) = 0$ for $\dim{X} = 2$; if not, we have to calculate modulo the Serre subcategory of $p$-torsion groups in the following), we have $E_2^{p,q} = 0$ modulo $p$-torsion for $q \leq 3$, and hence the result $E^n = 0$ for $0 \leq n \leq 3$ follows from the sequences
\[
    0 \to E_2^{n,0} \to E^n \to E_2^{0,n}
\]
for $n = 1, 2, 3$ which are exact modulo $p$-torsion.
\end{proof}

\begin{lemma} \label[lemma]{lemma:CodimensionOfPreimage}
If $f: X \to Y$ is a flat morphism of locally Noetherian schemes and $Z \hookrightarrow Y$ is a closed immersion of codimension $\geq c$, then also the base change $Z' := Z \times_Y X \hookrightarrow X$ is a closed immersion of codimension $\geq c$.
\end{lemma}
\begin{proof}
Without loss of generality, let $Z' = \overline{\{z'\}}$ be irreducible.  Since all involved schemes are locally Noetherian and $f$ is flat, by~\cite{Görtz-Wedhorn}, p.~464, Corollary~14.95 ($f: X \to Y$, $y = f(x)$, then the codimension of $\overline{\{x\}}$ is $\geq$ the codimension of $\overline{\{f(x)\}}$), we have $\codim_X Z' \geq \codim_Y \overline{f(z')}$.  But $\overline{f(z')} \subseteq \overline{Z} = Z$, so $\codim_Y \overline{f(z')} \geq \codim_Y Z \geq c$, hence the result.
\end{proof}

\begin{lemma} \label[lemma]{lemma:NS=Z}
Let $X$ be a normal scheme and $\Ccal/X$ a smooth proper relative curve.  Then there is an exact sequence
\[
    0 \to \PPic^0_{\Ccal/X} \to \PPic_{\Ccal/X} \to \Z \to 0.
\]
\end{lemma}
\begin{proof}
Without loss of generality, assume $X$ connected (as all schemes are of finite type over a field, so there are only finitely many connected components all of which are open: Every connected component is closed, and they are finite in number, so they are open).  Let $g = 1 - \chi_{\Ocal_\Ccal}$ be the genus of $\Ccal/X$ (well-defined because of~\cref{lemma:Euler characteristic locally constant}).  Consider
\begin{align*}
    \deg: \PPic_{\Ccal/X} &\to \Z, \\
    \Pic(\Ccal \times_X Y)/\pr_2^*\Pic(Y) \ni \Lcal &\mapsto \chi_\Lcal(y) - (1-g), \quad y \in Y;
\end{align*}
this is a well-defined morphism of Abelian sheaves on the small étale site of $X$ because of~\cref{lemma:Euler characteristic locally constant}.

Now the statement follows from~\cite{Conrad-pic}, p.~3\,f., Proposition~4.1 and Theorem~4.4.
\end{proof}

Next we construct a Leray spectral sequence for étale cohomology with supports.
\begin{theorem} \label{thm:LeraySSWithSupports}
If $i: Z \hookrightarrow Y$ is a closed immersion and $\pi: X \to Y$ is a morphism,
\[\xymatrix{
Z' \ar@{^{(}->}[r]^{i'} \ar[d] & X \ar[d]^\pi\\
Z  \ar@{^{(}->}[r]^{i}         &Y}\]
there is a $E_2$-spectral sequence for étale sheaves $\Fcal$
\[
    \H^p_Z(Y, \R^q\pi_*\Fcal) \Rightarrow \H^{p+q}_{Z'}(X, \Fcal),
\]
where $i': Z' \hookrightarrow X$ is the fibre product $\mathrm{pr}_2: Z \times_Y X \hookrightarrow X$.
\end{theorem}
\begin{proof}
This is the Grothendieck spectral sequence for the composition of functors generalising the Leray spectral sequence~\cite{MilneÉtaleCohomology}, p,~89, Theorem~III.1.18\,(a)
\begin{align*}
    F: &\Fcal \mapsto \pi_*\Fcal \\
    G: &\Fcal \mapsto \H^0_Z(Y, \Fcal),
\end{align*}
since
\begin{align*}
    (GF)(\Fcal) &= \H^0_Z(Y, \pi_*\Fcal) \\
                &= \ker((\pi_*\Fcal)(Y) \to (\pi_*\Fcal)(Y \setminus Z)) \\
                &= \ker(\Fcal(X) \to \Fcal(\pi^{-1}(Y \setminus Z))) \\
                &= \ker(\Fcal(X) \to \Fcal(X \setminus Z')) \\
                &= \H^0_{Z'}(X,\Fcal).
\end{align*}
We have to check that $\pi_*(-)$ maps injectives to $\H^0_Z(Y, -)$-acyclics.  Then~\cite{Weibel}, p.~150\,f., Theorem~5.8.3 establishes the existence of the spectral sequence.

Injective sheaves $\Ical$ are flabby (defined in~\cite{MilneÉtaleCohomology}, p.~87, Example~III.1.9\,(c)) and $\pi_*$ maps flabby sheaves to flabby sheaves (\cite{MilneÉtaleCohomology}, p.~89, Lemma~III.1.19).  Therefore, it follows from the long exact localisation sequence~\cite{MilneÉtaleCohomology}, p.~92, Proposition~III.1.25
\begin{align*}
    0 &\to \H^0_Z(Y,\pi_*\Ical) \to \H^0(Y,\pi_*\Ical) \to \H^0(Y \setminus Z,\pi_*\Ical) \\
      & \to \H^1_Z(Y,\pi_*\Ical) \to \H^1(Y,\pi_*\Ical) \to \H^1(Y \setminus Z,\pi_*\Ical) \\
      & \to \H^2_Z(Y,\pi_*\Ical) \to \H^2(Y,\pi_*\Ical) \to \H^2(Y \setminus Z,\pi_*\Ical) \to \ldots
\end{align*}
and $\H^p(Y,\pi_*\Ical) = 0 = \H^p(Y \setminus Z,\pi_*\Ical)$ for $p > 0$ that $\H^q_Z(Y,\pi_*\Ical) = 0$ for $q > 1$.  For $\H^1_Z(Y,\pi_*\Ical) = 0$, it remains to show that $\H^0(Y,\pi_*\Ical) \to \H^0(Y \setminus Z,\pi_*\Ical)$ is surjective.  For this, setting $j: U = X \setminus Z' \hookrightarrow X$, apply $\Hom(-,\Ical)$ to the exact sequence $0 \to j_!\Ocal_U \to \Ocal_X$ ($U = X \setminus Z'$) and get
\[
    \Ical(X) = \Hom(\Ocal_X,\Ical) \twoheadrightarrow \Hom(j_!\Ocal_U, \Ical) = \Hom(\Ocal_U,\Ical|_U) = \Ical(U),
\]
the arrow being surjective since $\Ical$ is injective.
\end{proof}

\begin{lemma} \label[lemma]{lemma:VerschwindungfuerCodim2}
Let $X/k$ be a smooth variety and $\pi: \Ccal \to X$ a smooth proper relative curve which admits a section $s: X \to \Ccal$.  Let $Z \hookrightarrow X$ be a reduced closed subscheme of codimension $\geq 2$.  Assume $\dim{X} \leq 2$.  Then
\[
    \H^i_Z(X, \PPic_{\Ccal/X}) = 0\quad\text{for $i \leq 2$}.
\]
For $\dim{X} > 2$, this holds at least up to $p$-torsion.
\end{lemma}
\begin{proof}
By~\cref{thm:RqGm=0q>1} and the Leray spectral sequence with supports $\H^p_Z(Y, \R^q\pi_*\Fcal) \Rightarrow \H^{p+q}_{Z'}(X, \Fcal)$ (see~\cref{thm:LeraySSWithSupports}), we get a long exact sequence
\begin{equation} \label{eq:lesSS}
    0 \to E_2^{1,0} \to E^1 \to E_2^{0,1} \to E_2^{2,0} \to E^2 \to E_2^{1,1} \to E^{3,0} \to E^{3} \to E_2^{2,1} \to E^{4,0} \to E^4.
\end{equation}
But $E_2^{i,0} = \H^i_Z(X, \G_m) = 0$ for $i \leq 3$ by~\cref{lemma:HiZGm=0} (for $i=3$ at least modulo $p$-torsion).  Therefore the long exact sequence~\eqref{eq:lesSS} yields isomorphisms
\[
    E^i \to E_2^{i-1,1}
\]
for $i \leq 2$, but $E^i = \H^i_{Z'}(\Ccal, \G_m) = 0$ for $i \leq 3$, again by~\cref{lemma:HiZGm=0} and~\cref{lemma:CodimensionOfPreimage}, hence $\H^i_Z(X, \R^1\pi_*\G_m) = E_2^{i-1,1} = 0$ for $i \leq 2$ from~\eqref{eq:lesSS}. For the vanishing of $E_2^{2,1}$ note that
\[
    0 = E^{3} \to E_2^{2,1} \to E^{4,0} \to E^4
\]
is exact by~\eqref{eq:lesSS}, but the latter map is $\pi^*: \H^4_Z(X, \G_m) \hookrightarrow \H^4_{Z'}(\Ccal,\G_m)$, which is injective as $\pi$ admits a section.
\end{proof}

\begin{lemma} \label[lemma]{lemma:konstante Garben}
Let $X$ be geometrically unibranch (e.\,g.\ normal) and $i: U \hookrightarrow X$ dominant.  Then for any constant sheaf $A$, one has $A \isoto i_*i^*A$ as étale sheaves.
\end{lemma}
\begin{proof}
See~\cite{SGA43}, p.~25\,f., IX Lemme~2.14.1.
\end{proof}

\begin{lemma} \label[lemma]{lemma:Kohomologie konstanter Garben}
Let $X$ be a connected normal Noetherian scheme with generic point $\eta$.  Then $\H^p(X,\Q) = 0$ for all $p > 0$ and $\H^1(X,\Z) = 0$.
\end{lemma}
\begin{proof}
Denote the inclusion of the generic point by $g: \{\eta\} \hookrightarrow X$.  Since $X$ is connected normal, by~\cref{lemma:konstante Garben}
\[
    A \isoto g_*g^*A.
\]

As $\R^qg_*(g^*\Q) = 0$ for $q > 0$ (since Galois cohomology is torsion and $\Q$ is uniquely divisible; then use~\cite{MilneÉtaleCohomology}, p.~88, Proposition~III.1.13), the Leray spectral sequence
\[
    \H^p(X,\R^qg_*g^*\Q) \Rightarrow \H^{p+q}(\eta, \Q)
\]
degenerates to $\H^p(X,\Q) = \H^p(\eta,\Q)$, which is trivial as, again, Galois cohomology is torsion and $\Q$ is uniquely divisible.

Similarly, as $\R^1g_*(g^*\Z) = 0$ (since $\Z$ carries the trivial Galois action and homomorphisms from profinite groups to discrete groups have finite image, and $\Z$ has no nontrivial finite subgroups; again, use~\cite{MilneÉtaleCohomology}, p.~88, Proposition~III.1.13), the Leray spectral sequence gives $\H^1(X,\Z) = \H^1(\eta,\Z) = 0$.
\end{proof}

\begin{remark}
If $X$ is not normal, $\H^i(X,\Z) \neq 0$ in general:\footnote{\url{http://mathoverflow.net/questions/84414/etale-cohomology-with-coefficients-in-the-integers}}  Consider $X = \Spec(k[X,Y]/(Y^2 - (X^3+X^2)))$, the normalisation morphism $\pi: \A_k^1 \to X$ and the inclusion $i: \{x\} \hookrightarrow X$ of $x = (0,0)$.  There is a short exact sequence of étale sheaves on $X$
\[
    0 \to \Z_X \to \pi_* \Z_{\A_k^1} \to i_*\Z_{\{x\}} \to 0.
\]
Taking the long exact cohomology sequence yields $\H^1(X,\Z_X) = \Z$.
\end{remark}

\begin{proposition} \label[proposition]{prop:EtalePi1UndH1}
Let $X$ be a connected Noetherian scheme and $\bar{x}$ a geometric point.  Then
\[
    \H^1(X,\Q/\Z) = \Hom_{cont}(\piet(X,\bar{x}),\Q/\Z)
\]
and
\[
    \H^1(X,\Z_\ell) = \Hom_{cont}(\piet(X,\bar{x}),\Z_\ell).
\]
\end{proposition}
\begin{proof}
This follows from~\cite{LeiFuÉtaleCohomologyTheory}, p.~245, Proposition 5.7.20 (for a connected Noetherian scheme $X$ and a \emph{finite} Abelian group $G$, one has $\H^1(X,G) = \Hom_{cont}(\piet(X,\bar{x}),G)$), via passing to the colimit over $G_n = \frac{1}{n}\Z/\Z$, or passing to the limit over $G_n = \Z/\ell^n$, respectively.
\end{proof}

\begin{lemma} \label[lemma]{lemma:VerschwindungfuerCodim2Pic0}
Let $X/k$ be a smooth variety and $\Ccal/X$ a smooth proper relative curve.  Assume $\dim{X} \leq 2$.  Let $Z \hookrightarrow X$ be a reduced closed subscheme of codimension $\geq 2$.  Then
\[
    \H^i_Z(X, \PPic^0_{\Ccal/X}) = 0\quad\text{for $i \leq 2$}.
\]
If $\dim{X} > 2$, this holds at least up to $p$-torsion.
\end{lemma}
\begin{proof}
Taking the long exact sequence associated to the short exact sequence of~\cref{lemma:NS=Z} with respect to $\H^i_Z(X,-)$, by~\cref{lemma:VerschwindungfuerCodim2}, it suffices to show that $\H_Z^i(X,\Z) = 0$ for $i = 0,1,2$.

For this, consider the long exact sequence
\[
    \ldots \to \H^i_Z(X,\Z) \to \H^i(X,\Z) \to \H^i(X \setminus Z, \Z) \to \ldots
\]
It suffices to show that $\H^i(X,\Z) \to \H^i(X \setminus Z, \Z)$ is an isomorphism for $i = 0,1$ and an injection for $i = 2$.

For $i = 0$, this map is $\Z \isoto \Z$.

For $i = 1$, both groups are equal to $0$ by~\cref{lemma:Kohomologie konstanter Garben}.

For $i = 2$, this map is $\H^2(X,\Z) \to \H^2(X \setminus Z, \Z)$.  Consider the long exact sequence associated to
\[
    0 \to \Z \to \Q \to \Q/\Z \to 0.
\]

In the following, we omit the base point for the fundamental groups.  Since $X$ is smooth over a field and hence normal, $\H^2(X,\Z) = \H^1(X,\Q/\Z) = \Hom_{cont}(\piet(X),\Q/\Z)$. (For the first equality, use the long exact sequence associated to $0 \to \Z \to \Q \to \Q/\Z \to 0$ and $\H^p(X,\Q) = 0$ for $p > 0$ by~\cref{lemma:Kohomologie konstanter Garben}.  For the second equality use~\cref{prop:EtalePi1UndH1}.)  Since $\piet(X \setminus Z) \to \piet(X)$ is an isomorphism because $Z$ is of codimension $\geq 2$ (by Zariski-Nagata purity~\cite{SGA1}, Exp.~X, Corollaire~3.3), $\H^2(X,\Z) \to \H^2(X \setminus Z, \Z)$ is an isomorphism:
\begin{align*}\xymatrix{
\H^2(X,\Z)              \ar[r]^{\kern-1cm\iso} \ar[d]   & \Hom_{cont}(\piet(X),\Q/\Z) \ar[d]^\iso \\
\H^2(X \setminus Z, \Z) \ar[r]^{\kern-1cm\iso}          & \Hom_{cont}(\piet(X \setminus Z),\Q/\Z)
}\end{align*}
\end{proof}

In the following, \textbf{assume} that the \textbf{vanishing condition}~\eqref{eq:vanishingcondition} $\H^i_Z(X,\Acal) = 0$ for $Z \hookrightarrow X$ a reduced closed subscheme of codimension $\geq 2$ and $i \leq 2$ is satisfied (at least up to $p$-torsion; if we do not want to consider $p$-torsion, calculate modulo the Serre subcategory of $p$-torsion groups).

\begin{lemma} \label[lemma]{lemma:cohomology of limit}
Let $I$ be a filtered category and $(i \mapsto X_i)$ a contravariant functor from $I$ to schemes over $X$.  Assume that all schemes are quasi-compact and that the transition maps $X_i \leftarrow X_j$ are affine.  Let $X_\infty = \varprojlim X_i$, and, for a sheaf $\Fcal$ on $X_{\mathrm{\acute{e}t}}$, let $\Fcal_i$ and $\Fcal_\infty$ be its inverse images on $X_i$ and $X_\infty$ respectively.  Then
\[
    \varinjlim \H^p((X_i)_{\mathrm{\acute{e}t}}, \Fcal_i) \isoto \H^p((X_\infty)_{\mathrm{\acute{e}t}}, \Fcal_\infty).
\]

Assume the $X_i \subseteq X$ are open, the transition morphisms are affine and all schemes are quasi-compact.  Let $Z \hookrightarrow X$ be a closed subscheme.  Then
\[
    \varinjlim \H^p_{Z \cap X_i}((X_i)_{\mathrm{\acute{e}t}}, \Fcal_i) \isoto \H^p_{Z \cap X_\infty}((X_\infty)_{\mathrm{\acute{e}t}}, \Fcal_\infty).
\]
\end{lemma}
\begin{proof}
See~\cite{MilneÉtaleCohomology}, p.~88, Lemma~III.1.16 for the first statement.  The second one follows from the first, the long exact localisation sequence (note that the morphisms $(X \setminus Z) \cap X_i \leftarrow (X \setminus Z) \cap X_j$ are affine as well since they are base changes of affine morphisms) and the five lemma.
\end{proof}

\begin{corollary} \label[lemma]{cor:Colimes über alle offenen}
We have
\[
    \varinjlim_U \H^p(U, \Fcal|_U) \isoto \H^p(K, \Fcal_\eta),
\]
the colimit with respect to the restriction maps, where $U$ runs through the non-empty standard affine open subschemes $D(f_i)$ of an non-empty affine open subscheme $\Spec(A) \subseteq X$ and $\{\eta\} = \Spec(K)$.
\end{corollary}
\begin{proof}
Set $X_i = U$ in~\cref{lemma:cohomology of limit} and note that $\varprojlim U = \Spec(K)$.
\end{proof}

\begin{lemma}[Excision of codimension $\geq 2$ subschemes] \label[lemma]{lemma:ExcisionOfCodimgeq2subschemes}
One can excise subschemes $Z \hookrightarrow Y$ of codimension $\geq 2$ in $X$:
\begin{align} \label{eq:Ausschneidung und Kerne}
    \ker\left(\H^1(U,\Acal) \to \H^2_Y(X,\Acal)\right) = \ker\left(\H^1(U,\Acal) \to \H^2_{Y \setminus Z}(X \setminus Z, \Acal|_{X \setminus Z})\right).
\end{align}
\end{lemma}
\begin{proof}
From the long exact localisation sequence for cohomology with supports~\cite{MilneÉtaleCohomology}, p.~92, Remark~III.1.26
\[
    \ldots \to \H^p_Z(X,\Acal) \to \H^p_Y(X,\Acal) \to \H^p_{Y \setminus Z}(X \setminus Z,\Acal) \to \ldots
\]
and from the vanishing condition one gets the injectivity
\begin{align} \label{eq:Ausschneidung und Purity}
    0 \to \H^2_Y(X,\Acal) \hookrightarrow \H^2_{Y \setminus Z}(X \setminus Z, \Acal|_{X \setminus Z}),
\end{align}
hence by the kernel-cokernel exact sequence the claim.
\end{proof}

We will use a Mayer-Vietoris sequence for cohomology with supports.
\begin{lemma} \label[lemma]{thm:MayerVietorisWithSupports}
Let $Y_1$ and $Y_2$ be closed subschemes of $X$ and $\Fcal$ a sheaf on $X$.  Then there is a long exact sequence of cohomology with supports
\[
    \ldots \to \H^i_{Y_1 \cap Y_2}(X,\Fcal) \to \H^i_{Y_1}(X,\Fcal) \oplus \H^i_{Y_2}(X,\Fcal) \to \H^i_{Y_1 \cup Y_2}(X,\Fcal) \to \ldots
\]
\end{lemma}
\begin{proof}
Choosing an injective resolution $0 \to \Fcal \to \Ical^\bullet$, this follows from the exact sequence of complexes
\[
    0 \to \Gamma_{Y_1 \cap Y_2}(X,\Ical^\bullet) \to \Gamma_{Y_1}(X,\Ical^\bullet) \oplus \Gamma_{Y_2}(X,\Ical^\bullet) \to \Gamma_{Y_1 \cup Y_2}(X,\Ical^\bullet) \to 0
\]
in the usual way.
\end{proof}

\begin{lemma}
Let $Y \hookrightarrow X$ be a closed subscheme with open complement $U = X \setminus Y$ and with all of its irreducible components of codimension $1$ in $X$.  Denote the finitely many irreducible components of $Y$ by $(Y_i)_{i=1}^n$.  Then
\begin{equation} \label{eq:bla}
    \ker\left(\H^1(U,\Acal) \to \H^2_Y(X,\Acal)\right) = \ker\Big(\H^1(U,\Acal) \to \bigoplus_{i=1}^n \H^2_{Y_i \setminus Z_i}(X \setminus Z_i, \Acal)\Big)
\end{equation}
with certain closed subschemes $Z_i \hookrightarrow Y_i$.
\end{lemma}
\begin{proof}
Using~\cref{lemma:ExcisionOfCodimgeq2subschemes}, excise the intersections $Y_i \cap Y_j$ for $i \neq j$ (they are of codimension $\geq 2$ in $X$ since our schemes are catenary as they are varieties by~\cite{Liu2006}, p.~338, Corollary~8.2.16 and
\[
   2 = 1 + 1 \leq \codim(Y_i \cap Y_j \hookrightarrow Y) + \codim(Y \hookrightarrow X) = \codim(Y_i \cap Y_j \hookrightarrow X)).
\]
Now, by a repeated application of the Mayer-Vietoris sequence with supports in~\cref{thm:MayerVietorisWithSupports}, one gets the claim.
\end{proof}

\begin{lemma} \label[lemma]{lemma:UFD}
Let $X$ be a regular variety over a finitely generated field and $x \in X^{(1)}$.  Then there is an open affine subscheme $\Spec{A} = X_0 \subseteq X$ containing $x$ such that $A$ is a unique factorisation domain.
\end{lemma}
\begin{proof}
The class group $\mathrm{Cl}(X)$ is finitely generated by~\cite{KahnGroupeDesClasses}, p.~396, Corollaire~2.  Excise the finite set of generators of the Picard group using~\cite{HartshorneAG}, p.~133, Proposition~II.6.5.  (If $\overline{\{x\}}$ is one of the generators, replace it by the moving lemma~\cite{Liu2006}, p.~380, Proposition~9.1.11.)  Then the rest of the variety has trivial class group.  Now take an open affine subset of the remaining scheme containing $x$.  It is normal and has trivial class group, so by~\cite{HartshorneAG}, p.~131, Proposition~II.6.2 it is a unique factorisation domain.
\end{proof}

\begin{lemma}
One has
\begin{equation} \label{eq:ShaNurNochAbgPkteImTraeger}
    \ker\left(\H^1(U,\Acal) \to \H^2_Y(X,\Acal)\right) = \ker\Big(\H^1(U,\Acal) \to \bigoplus_{i=1}^n \H^2_{\{x_i\}}(X \setminus \tilde{Z}_i, \Acal)\Big)
\end{equation}
with the $x_i \in X \setminus \tilde{Z}_i$ closed points and generic points of the $Y_i$ and $\tilde{Z}_i \hookrightarrow X$ certain subschemes.
\end{lemma}
\begin{proof}
This follows basically by excising (using~\eqref{eq:Ausschneidung und Kerne}) everything except the generic points of the $Y_i$ in~\eqref{eq:bla}.  The only technical difficulty is that for applying~\cref{lemma:cohomology of limit} one has to make sure that the transition maps are affine.

Fix an irreducible component $Y_i$ of $Y$ and call it $Y$ with $x$ its generic point and with corresponding $Z := Z_i$.  We want to construct an injection
\begin{equation} \label{eq:injection}
    \H^2_{Y \setminus Z}(X \setminus Z, \Acal) \hookrightarrow \H^2_{\{x\}}(X \setminus \tilde{Z}, \Acal).
\end{equation}

Using~\cref{lemma:UFD}, choose an affine open $X_0 := \Spec{A} \subset X \setminus Z$ containing the point $x$ such that $A$ is a unique factorisation domain.  One has $\H^q_Y(X,\Acal) \hookrightarrow \H^q_{Y \cap X_0}(X_0, \Acal)$: $X \setminus X_0$ is a closed subscheme $V \hookrightarrow X$ such that $V \cap Y \hookrightarrow Y$ is of codimension $\geq 1$ and hence (since varieties are catenary by~\cite{Liu2006}, p.~338, Corollary~2.16) $V \cap Y \hookrightarrow X$ of codimension $\geq 2$, so we conclude by excision~\eqref{eq:Ausschneidung und Kerne}.  We construct a sequence $(X_i)_{i=0}^\infty$ of standard affine open subsets $X_i = D(f_i) \subset X_0$, $X_{i+1} \subset X_i$, all of them containing the point $x$ such that
\begin{align*}
    \H^q_{Y \cap X_i}(X_i,\Acal) &\hookrightarrow \H^q_{Y \cap X_{i+1}}(X_{i+1}, \Acal), \quad\text{and thus by~\eqref{eq:Ausschneidung und Kerne}}\\
    \ker\left(\H^1(U,\Acal) \to \H^2_Y(X,\Acal)\right) &= \ker\left(\H^1(U,\Acal) \to \H^2_{Y \cap X_i}(X_i, \Acal)\right),
\end{align*}
and such that $\varprojlim_j X_i \cap Y = \{x\}$.  Then~\eqref{eq:injection} follows from~\cref{lemma:cohomology of limit}.

Since $X$ is countable, one can choose an enumeration $(Z_i)_{i=1}^\infty$ of the closed integral subschemes of codimension $1$ of $X_0$ not equal to $X_0 \cap Y$.

Given $X_i = D(f_i)$, take $f \in \Spec{A_{f_i}}$ such that $V(f) = Z_{i+1}$.  (By the converse of Krull's Hauptidealsatz~\cite{Eisenbud}, p.~233, Corollary~10.5, since $Z_{i+1}$ is of codimension $1$ in $X_0$, there is an $f \in A$ such that $Z_{i+1} \subseteq V(f)$ is minimal.  Since $X_0$ is a unique factorisation domain, one can assume $f$ prime, hence $Z_{i+1} = V(f)$ by codimension reasons.)  Then $V(f) \cap (X^{(1)} \cap Y) = \emptyset$ and $V(f) \cap Y \subsetneq Y$ has codimension $\geq 1$ in $Y$, hence (again, varieties being catenary) $V(f) \cap Y$ has codimension $\geq 2$ in $X$.  Therefore one can apply~\eqref{eq:Ausschneidung und Purity} to yield an injection
\[
    \H^2_{Y \cap X_i}(X_i,\Acal) \hookrightarrow \H^2_{(Y \cap X_i) \setminus V(f)}(X_i \setminus (Y \cap V(f)),\Acal).
\]
Now, by excision (\cite{MilneÉtaleCohomology}, p.~92, Proposition~III.1.27) of $V(f)$, one has
\[
    \H^2_{(Y \cap X_i) \setminus V(f)}(X_i \setminus (Y \cap V(f)),\Acal) \isoto \H^2_{(Y \cap X_i) \setminus V(f)}(X_i \setminus V(f),\Acal).
\]
Set $X_{i+1} = X_i \setminus V(f) = D(f_if)$, $f_{i+1} = f_if$.\\

Apply this to the direct summands in~\eqref{eq:bla}.
\end{proof}

\begin{lemma}
There is a short exact sequence
\begin{align}
    0 \to \H^1_Y(X,\Acal)/\Acal(U) \to \H^1(X,\Acal) \to \ker\left(\H^1(U,\Acal) \to \H^2_Y(X,\Acal)\right) \to 0. \label{eq:ShaUndKohomologieMitTraegern}
\end{align}
\end{lemma}
\begin{proof}
First note that the map $\H^1(X,\Acal) \to \ker(\ldots)$ in~\cref{thm:codim1Punktereichen} is well-defined since if $x \in \H^1(X,\Acal)$ is restricted to $\H^1(K,\Acal)$ via $g: \{\eta\} \hookrightarrow X$, its pullback to $K_x^{nr} = \Spec(\Quot(\Ocal_{X,x}^{sh}))$ factors as
\[
    \H^1(X,\Acal) \to \H^1(\Spec(\Ocal_{X,x}^{sh}), \Acal) \to \H^1(\Spec(\Quot(\Ocal_{X,x}^{sh})), \Acal),
\]
but the étale site of $\Spec(\Ocal_{X,x}^{sh})$ is trivial.

Let $\emptyset \neq U \hookrightarrow X$ be open with reduced closed complement $Y \hookrightarrow X$.  Then one has the long exact localisation sequence \cite{MilneÉtaleCohomology}, p.~92, Proposition~III.1.25
\[
    0 \to \H^0_Y(X,\Acal) \to \H^0(X,\Acal) \to \H^0(U,\Acal) \to \H^1_Y(X,\Acal) \to \H^1(X,\Acal) \to \H^1(U,\Acal) \to \H^2_Y(X,\Acal) \to \ldots.
\]
Because of the injectivity of $\H^0(X,\Acal) \hookrightarrow \H^0(U,\Acal)$ (If two sections of $\Acal/X$ coincide on $U$ open dense, they agree on $X$ since $\Acal$ is separated and $X$ reduced), the exactness of the sequence yields $\H^0_Y(X,\Acal) = 0$, and hence the short exact sequence~\eqref{eq:ShaUndKohomologieMitTraegern}.
\end{proof}

\begin{lemma} \label[lemma]{lemma:Surjectivity} \label{page:Surjectivity}
The homomorphism~\eqref{eq:Codim1} in~\cref{thm:codim1Punktereichen}\,(b) is surjective.
\end{lemma}
\begin{proof}
Now for the proof of~\cref{thm:codim1Punktereichen}, at least for prime-to-$p$ torsion (confer~\cref{lemma:VerschwindungfuerCodim2}).  If the vanishing condition~\eqref{eq:vanishingcondition} is only satisfied for the prime-to-$p$ torsion, in the whole proof the following computation only holds modulo $p$-torsion.

Applying excision in the form of~\cite{MilneÉtaleCohomology}, p.~93, Corollary~III.1.28 to~\eqref{eq:ShaNurNochAbgPkteImTraeger} (using that the $x_i \in X \setminus \tilde{Z}_i$ are closed points), one gets
\[
    \ker\left(\H^1(U,\Acal) \to \H^2_Y(X,\Acal)\right) = \ker\left(\H^1(U,\Acal) \stackrel{(r_i)}{\to} \bigoplus_{i=1}^n\H^2_{\{x_i\}}(X_{x_i}^h, \Acal)\right).
\]
Now $(r_i)$ factors as
\[
    r_i: \H^1(U,\Acal) \stackrel{j_{x_i}^*}{\to} \H^1(K_{x_i}^h, \Acal) \stackrel{\delta_i}{\to} \H^2_{\{x_i\}}(X_{x_i}^h, \Acal),
\]
where the latter map is the boundary map of the localisation sequence associated to the discrete valuation ring $\Ocal_{X,x_i}^h$ ($X$ is normal as it is smooth over a field, $x_i$ is a codimension-$1$ point, and the Henselisation of a normal ring is normal again by~\cite{LeiFuÉtaleCohomologyTheory}, p.~106, Proposition~2.8.10, and normal rings are $(\R_1)$)
\[
    X_{x_i}^h = \Spec(\Quot(\Ocal_{X,x_i}^h)) \cup \{x_i\}
\]
($x_i$ is the closed point [Henselisation preserves residue fields], and $\Spec(K_{x_i}^h)$ is the generic point of $X_{x_i}^h$).
One has $\H^1(X_{x_i}^h, \Acal) = \H^1(\kappa(x_i),\Acal)$ by~\cite{MilneÉtaleCohomology}, p.~116, Remark~III.3.11\,(a) and the inflation-restriction exact sequence
\[
    0 \to \H^1(X_{x_i}^h, \Acal) \stackrel{\infl}{\to} \H^1(K_x^h,\Acal) \stackrel{\res}{\to} \H^1(K_x^{nr},\Acal).
\]
Since $\Ocal_{X,x_i}^{sh}$ is a discrete valuation ring (as above; the strict Henselisation of a normal ring is normal again by~\cite{LeiFuÉtaleCohomologyTheory}, p.~111, Proposition~2.8.18), the valuative criterion of properness~\cite{EGAII}, p.~144\,f., Théorème~(7.3.8) gives us $\Acal(K_{x_i}^{nr}) = \Acal(X_{x_i}^{sh})$.  Hence one can write $j_i^* = \infl$ as the inflation
\[
    \H^1(\kappa(x_i), \Acal(X_{x_i}^{sh})) \hookrightarrow \H^1(K_{x_i}^h, \Acal),
\]
so the cokernel of $j_{x_i}^*$ injects into $\H^1(K_{x_i}^{nr}, \Acal)$ and in $\H^2_{\{x_i\}}(X_{x_i}^h, \Acal)$, so we get
\begin{align*}
    \ker\Big(\H^1(U,\Acal) \to \bigoplus_{i=1}^n\coker(j_{x_i})\Big) &\isoto \ker\Big(\H^1(U,\Acal) \to \bigoplus_{i=1}^n\H^2_{\{x_i\}}(X_{x_i}^h, \Acal)\Big) \\
    \ker\Big(\H^1(U,\Acal) \to \bigoplus_{i=1}^n\coker(j_{x_i})\Big) &\isoto \ker\Big(\H^1(U,\Acal) \to \bigoplus_{i=1}^n\H^1(K_{x_i}^{nr}, \Acal)\Big)
\end{align*}
and hence
\[
    \ker\left(\H^1(U,\Acal) \to \H^2_Y(X,\Acal)\right) = \ker\Big(\H^1(U,\Acal) \stackrel{(r_i)}{\to} \bigoplus_{i=1}^n\H^1(K_{x_i}^{nr}, \Acal)\Big).
\]

Taking the limit over all $Y$ (choose an enumeration $(Y_i)_{i=1}^\infty$ of the integral closed subschemes of codimension $1$ of $X$ [$X$ is countable]) yields by~\cref{cor:Colimes über alle offenen} and the fact that $(\mathbf{Ab})$ satisfies $\mathrm{(AB5)}$
\[
    \ker\Big(\H^1(K,\Acal) \to \varinjlim_{Y}\H^2_Y(X,\Acal)\Big) = \ker\Big(\H^1(K,\Acal) \to \bigoplus_{x \in X^{(1)}}\H^1(K_{x}^{nr}, \Acal)\Big).
\]
Now~\eqref{eq:ShaUndKohomologieMitTraegern} gives us an exact sequence
\[
    0 \to \varinjlim\H^1_Z(X,\Acal)/\Acal(U) \to \H^1(X,\Acal) \to \ker\Big(\H^1(K,\Acal) \to \bigoplus_{x \in X^{(1)}}\H^1(K_{x}^{nr}, \Acal)\Big) \to 0. \qedhere
\]
\end{proof}

\begin{lemma} \label[lemma]{lemma:Injectivity}
The homomorphism~\eqref{eq:Codim1} in~\cref{thm:codim1Punktereichen}\,(b) is injective.
\end{lemma}
\begin{proof}
The surjective homomorphism factors as (with the isomorphism from~\cref{thm:ShaSequenzAllePunkte})
\begin{align*}
    \H^1(X,\Acal) &\isoto \ker\Big(\H^1(K,\Acal) \to \prod_{x \in X}\H^1(K_{x}^{nr}, \Acal)\Big)     \\
                  &\hookrightarrow \ker\Big(\H^1(K,\Acal) \to \bigoplus_{x \in X^{(1)}}\H^1(K_{x}^{nr}, \Acal)\Big),
\end{align*}
so the claim follows. \label{ProofOfCodim1End}
\end{proof}

One can replace the strict Henselisation $\Ocal_{X,x}^{sh}$ of $\Ocal_{X,x}$ by its completion $\hat{\Ocal}_{X,x}^{sh}$ (respectively by their quotient fields) in the case of $x \in X^{(1)}$:
\begin{lemma} \label[lemma]{lemma:HenselisationAndCompletion}
Let $(A,\mathfrak{m})$ be a discrete valuation ring of a variety.  Let $Z$ be a smooth proper scheme of finite type over $A^{sh}$.  The following are equivalent:

1. $Z$ has a point over $\Quot(A^{sh})$.

2. $Z$ has a point over $A^{sh}$.

3. $Z$ has points over $A^{sh}/\mathfrak{m}^nA^{sh} = \hat{A}^{sh}/\mathfrak{m}^n\hat{A}^{sh}$ for all $n \gg 0$.

4. $Z$ has a point over $\hat{A}^{sh}$.

5. $Z$ has a point over $\Quot(\hat{A}^{sh})$.

\noindent
Analogous statements hold for $A^{sh}$ replaced by $A^h$.
\end{lemma}

\begin{proof}
In 3, the equality $A^{sh}/\mathfrak{m}^nA^{sh} = \hat{A}^{sh}/\mathfrak{m}^n\hat{A}^{sh}$ holds by~\cite{Eisenbud}, p.~183, Theorem~7.1.

One has $1 \iff 2$ and $4 \iff 5$ by the valuative criterion for properness~\cite{EGAII}, p.~144\,f., Théorème~(7.3.8), note that we are in codimension $1$.

The implications $2 \implies 3$ and $4 \implies 3$ are trivial, and the implication $3 \implies 4$ holds since smooth implies formally smooth~\cite{EGAIV4}, Définition~(17.1.1).

Finally, the implication $4 \implies 2$ follows from Artin approximation, see~\cite{ArtinApproximation}, p.~26, Theorem~(1.10) resp.\ Theorem~(1.12), applied to the functor $T \mapsto Z(T)$ which is locally of finite presentation.
\end{proof}

\subsection{Relation between the Brauer group and the Tate-Shafarevich group}
In analogy with the conjectures of Birch and Swinnerton-Dyer, and of Artin and Tate, one could hope that the Tate-Shafarevich gorup $\Sha(\Acal/X)$ and the Brauer group $\Br(\Ccal)$ are finite.  Recall that the \emph{index} of a variety $X/k$ is the greatest common divisor over all residue degrees $[k(x):k]$ where $x$ runs through the closed points of $X$.  We can extend the results of Artin and Tate~\cite{Artin-Tate} (for a curve $X$) as follows to our more general setting:
\begin{theorem}[The Artin-Tate and the Birch-Swinnerton-Dyer conjecture]
Assume that we are in the situation of~\cref{cor:BrauerUndShaExakteSequenz} and that $X$ has an ample sheaf.  Then one has an exact sequence
\[
    0 \to K_2 \to \Br(X) \stackrel{\pi^*}{\to} \Br(\Ccal) \to \Sha(\PPic_{\Ccal/X}/X) \to K_3 \to 0
\]
in which the groups $K_i$ are annihilated by $\delta$, the index of the generic fibre $C/K$, e.\,g.\ $\delta = 1$ if $\Ccal/X$ has a section, and their prime-to-$p$ parts are finite, and $K_i = 0$ if $\pi$ has a section.  Here, $\Sha(\PPic_{\Ccal/X}/X)$ sits in a short exact sequence
\[
    0 \to \Z/d \to \Sha(\PPic_{\Ccal/X}^0/X) \to \Sha(\PPic_{\Ccal/X}/X) \to 0,
\]
where $d \mid \delta$.
\end{theorem}
Hence the finiteness of the ($\ell$-torsion of the) Brauer group of $\Ccal$ is equivalent to the finiteness of the ($\ell$-torsion of the) Brauer group of the base $X$ and the finiteness of the ($\ell$-torsion of the) Tate-Shafarevich group of $\PPic_{\Ccal/X}$.
\begin{proof}
Combining~\cref{cor:BrauerUndShaExakteSequenz} with $\R^1\pi_*\G_m = \PPic_{X/S}$ by~\cite{BLR}, p.~202\,f.\ and using~\cref{cor:BrFuerVarietyOverFiniteField} to identify $\H^2(X,\G_m)$ with $\Br(X)$ (and similarly for $\Ccal$) yields the exact sequence
\[
    0 \to K_2 \to \Br(X) \stackrel{\pi^*}{\to} \Br(\Ccal) \to \H^1(X, \PPic_{\Ccal/X}) \to K_3 \to 0
\]
with the prime-to-$p$ part of the $K_i$ finite, and $K_i = 0$ if $\pi$ has a section.

Now the long exact sequence associated to the short exact sequence in~\cref{lemma:NS=Z} yields the exact sequence
\[
    \H^0(X,\PPic_{\Ccal/X}) \to \H^0(X,\Z) \to \H^1(X, \PPic_{\Ccal/X}^0) \to \H^1(X, \PPic_{\Ccal/X}) \to \H^1(X,\Z) = 0,
\]
and $\H^1(X,\Z) = 0$ using~\cref{lemma:Kohomologie konstanter Garben}.  Now, choose a Weil divisor $D$ on the generic fibre $C/K$ of $\Ccal/X$ with degree $\delta$ the index of $C/K$.  By~\cref{lemma:finiteMorphismOfRelativeCurves} and~\cite{Conrad-pic}, p.~3, Proposition~4.1, $\bar{D}$ is a Weil divisor on $\Ccal$ of degree $\delta$, and its image under $\H^0(X,\PPic_{\Ccal/X}) \to \H^0(X,\Z) = \Z$ (here we use that $X$ is connected) is $\delta$.  Hence $\coker(\H^0(X,\PPic_{\Ccal/X}) \to \H^0(X,\Z) = \Z)$ is a quotient of $\Z/\delta$.
\end{proof}

\subsection{Descent of finiteness of $\Sha$ under generically étale alterations}
\begin{lemma} \label[lemma]{lemma:ShaEndlichenEllKorang}
Let $\ell$ be invertible on $X$.  Then the $\Z_\ell$-corank of $\Sha(\Acal/X)[\ell^\infty]$ is finite.
\end{lemma}
\begin{proof}
The short exact sequence of étale sheaves $0 \to \Acal[\ell^n] \to \Acal \to \Acal \to 0$ induces
\[
    0 \to \Acal(X)/\ell^n \to \H^1(X,\Acal[\ell^n]) \to \H^1(X,\Acal)[\ell^n] \to 0.
\]
From this, one sees that $\H^1(X,\Acal)[\ell]$ is finite as it is a quotient of $\H^1(X,\Acal[\ell])$ and $\Acal[\ell]$ is a constructible sheaf on $X$.  Hence $\Sha(\Acal/X)[\ell^\infty]$ is cofinitely generated.
\end{proof}

\begin{theorem}
Let $f: X' \to X$ b a morphism of regular schemes which is an $\ell'$-alteration for a prime $\ell$ invertible on $X$, i.\,e.\ $f$ is a proper, surjective, generically étale morphism of generical degree prime to $\ell$.  Let $\ell$ be some prime invertible on $X$.  If $\Acal$ is an Abelian scheme on $X$ such that the $\ell$-torsion of the Tate-Shafarevich group $\Sha(\Acal'/X')$ of $\Acal' := f^*\Acal = \Acal \times_X X'$ is finite, then the $\ell$-torsion of the Tate-Shafarevich group $\Sha(\Acal/X)$ is finite.
\end{theorem}
\begin{proof}
Without loss of generality, assume $X$ integral with generic point $\eta$.  Define $X'_\eta$ by the commutativity of the cartesian diagram
\begin{align}\label{eq:kommutatives Quadrat Sha}\xymatrix{
X'_\eta \ar@{^{(}->}[r]^{g'} \ar[d]^f & X' \ar[d]^f\\
\{\eta\}  \ar@{^{(}->}[r]^{g}           & X.}\end{align}

\textbf{Step 1:  {\boldmath$\H^1(X, f_*\Acal')[\ell^\infty]$} is finite.}  This follows from the low terms exact sequence
\[
    0 \to \H^1(X, f_*\Acal') \to \H^1(X', \Acal')
\]
associated to the Leray spectral sequence $\H^p(X, \R^qf_*\Acal') \Rightarrow \H^{p+q}(X', \Acal')$ and the finiteness of $\H^1(X', \Acal')[\ell^\infty] = \Sha(\Acal'/X')[\ell^\infty]$.

\textbf{Step 2:  The theorem holds if there is a trace morphism.}  Assume there is a trace morphism $f_*f^*\Acal \to \Acal$ such that the composition with the adjunction morphism
\[
    \Acal \to f_*f^*\Acal \to \Acal
\]
is multiplication with $\deg{f}$, where $\deg{f}$ is invertible on $X$.  Let $A = \H^1(X,\Acal)[\ell^\infty]$ and $B = \H^1(X,f_*\Acal')[\ell^\infty]$ and denote the induced morphisms on cohomology by $g: A \to B$ und $h: B \to A$.  By~\cref{lemma:ShaEndlichenEllKorang}, $A$ is cofinitely generated.  Since $B$ is a finite $\ell$-group by Step~1, there is an $N \in \N$ such that $\ell^N \cdot B = 0$.  Then one has $\ell^N g(A) = 0$, thus $\ell^N (h \circ g) = \ell^N[\deg{f}] = 0$ as an endomorphism of $A$.  As $A$ is cofinitely generated and $\ell^N \cdot \deg{f} \neq 0$, the finiteness of $A$ follows.

\textbf{Step 3:  Construction of the trace morphism for {\boldmath$X = \Spec{k}$} a field.}  Since $f$ is étale, by~\cite{LeiFuÉtaleCohomologyTheory}, p.~205, Proposition~5.5.1\,(i), $f_!$ is left adjoint to $f^*$ and by~\emph{loc.\ cit.}\,(iv), one has
\[
    (f_!\Fcal)_{\bar{x}} = \bigoplus_{x' \in X' \otimes_{\Ocal_X} \overline{k(x)}}\Fcal_{\bar{x}'}.
\]
Since $f$ is étale, proper and of finite type, by~\cite{LeiFuÉtaleCohomologyTheory}, p.~207, Proposition~5.5.2, one has $f_! = f_*$.  Hence, there is an adjunction morphism $f_*f^*\Acal = f_!f^*\Acal \to \Acal$, and we have to prove that
\[
    \Acal_{\bar{x}} \to (f_*f^*\Acal)_{\bar{x}} = (f_!f^*\Acal)_{\bar{x}} \to \Acal_{\bar{x}}
\]
is multiplication by $\deg{f}$.  Since one can assume $\bar{x} = \Spec{k} = X$ with $k$ separably closed is strictly Henselian, $X' = \amalg_{i=1}^{\deg{f}}\Spec{k}$.  Hence the morphisms are
\begin{align*}
    \Acal_{\bar{x}} \to (f_*f^*\Acal)_{\bar{x}}, \quad&s \mapsto (s)_{i=1}^{\deg{f}}, \\
    (f_!f^*\Acal)_{\bar{x}} \to \Acal_{\bar{x}}, \quad&(s_i)_{i=1}^{\deg{f}} \mapsto \textstyle\sum_{i=1}^{\deg{f}}s_i,
\end{align*}
and the claim is obvious.

\textbf{Step 4:  Construction of the trace morphism for {\boldmath$X$} arbitrary regular.}  One has to construct a commutative diagram
\begin{align} \label{eq:Diagramm Sha} \xymatrix{
\Acal \ar[r]^\iso \ar[d] & g_*g^*\Acal \ar[d]\\
f_*f^*\Acal \ar[r] \ar[d] & g_*g^*(f_*f^*\Acal) \ar[d]\\
\Acal  \ar[r]^\iso           & g_*g^*\Acal
}\end{align}
such that the composition of the vertical maps equals multiplication by $\deg{f}$.

The upper square comes from the functoriality of the adjunction morphism associated to the natural transformation of functors $\id \to g_*g^*$.  The crucial point is to construct the lower right morphism $g_*g^*(f_*f^*\Acal) \to g_*g^*\Acal$ corresponding to the trace morphism (it is not directly possible to apply $g_*g^*$ to a map $f_*f^*\Acal \to \Acal$ since it is not clear how to define the latter).  This map arises as follows:  Since the domain of $g$ is a field, by Step~3, there is a trace morphism
\[
    f_*f^*(g^*\Acal) = f_!f^*(g^*\Acal) \to g^*\Acal.
\]
By the commutativity of the diagram~\eqref{eq:kommutatives Quadrat Sha}, one has $g'^*f^*\Acal = f^*g^*\Acal$, giving us an arrow
\[
    f_*g'^*(f^*\Acal) \to g^*\Acal.
\]
We use the proper base change theorem for non-torsion sheaves in~\cite{DeningerProperBaseChange}, p.~231, Theorem~1.1.  The assumptions are:  Noetherian, $f$ proper, $X'$ excellent, $g'$ normal, i.\,e.\ flat with geometrically normal fibres: $g'$ is flat as it is the base change of the flat morphism $g: \{\eta\} \hookrightarrow X$.  So $g^*f_* = f_*g'^*$, and we get
\[
    g^*f_*(f^*\Acal) \to g^*\Acal.
\]
Finally, applying $g_*$, we get our trace morphism on the right hand side.

By Step~3, it is obvious that the composition of the vertical arrows on the right hand side is multiplication by $\deg{f}$.

Now the lower left arrow $f_*f^*\Acal \to \Acal$ can be defined by the commutativity of the lower square of~\eqref{eq:Diagramm Sha}, noting that $\Acal \isoto g_*g^*\Acal$ is invertible by~\cref{thm:PicIsNeronModel}.  It follows from the right hand side that the composition of the vertical arrows on the left hand side is multiplication by $\deg{f}$.
\end{proof}

\begin{corollary}
In the situation of~\cref{cor:BrauerUndShaExakteSequenz}, if $\Br(\Ccal)[\ell^\infty]$ is finite for $\ell$ invertible on $X$, $\Br(X)[\ell^\infty]$ is finite.
\end{corollary}
\begin{proof}
Obvious from~\cref{cor:BrauerUndShaExakteSequenz}.
\end{proof}

\subsection{Isogeny invariance of finiteness of $\Sha$}
\begin{theorem}
Let $X/k$ be proper, $\Acal$ and $\Acal'$ Abelian schemes over $X$ and $f: \Acal' \to \Acal$ an étale isogeny.  Let $\ell \neq \Char{k}$ be a prime.  Then $\Sha(\Acal/X)[\ell^\infty]$ is finite if and only if $\Sha(\Acal'/X)[\ell^\infty]$ is finite.
\end{theorem}
\begin{proof}
This follows from the long exact cohomology sequence associated to the short exact sequence of étale sheaves ($f$ is étale and surjective) $0 \to \ker(f) \to \Acal' \to \Acal \to 0$ and~\cite{MilneÉtaleCohomology}, p.~224\,f., Corollary~VI.2.8.
\end{proof}

\subsection{The Cassels-Tate pairing}
In this subsection, we construct in some cases of a higher dimensional basis $X$ a generalised Cassels-Tate pairing $\Sha(\Acal/X)[\ell^\infty] \times \Sha(\Acal^\vee)[\ell^\infty] \to \Q_\ell/\Z_\ell$ by reduction to the case of a surface and then to the case of a curve using ample hypersurface sections.

\begin{theorem}
Let $k = \F_q$ be a finite field with absolute Galois group $\Gamma = G_k$ and $X/k$ a smooth \emph{projective} geometrically connected variety of dimension $d$.  Let $\Acal/X$ be an Abelian scheme and $\ell \neq p = \Char{k}$ be prime.  Assume that there is a sequence of ample smooth geometrically connected hypersurface sections $C \hookrightarrow S \hookrightarrow Y_3 \hookrightarrow \ldots \hookrightarrow X$ with $S$ a surface and $C$ a curve such that the canonical homomorphism $\Acal(X)/\ell^n \to \Acal(C)/\ell^n$ is surjective.  Then there is a pairing with left and right kernels the divisible part
\[
    \Sha(\Acal/X)[\ell^\infty] \times \Sha(\Acal^\vee)[\ell^\infty] \to \Q_\ell/\Z_\ell.
\]
\end{theorem}

\begin{proof}
Using the hypersurface sections and the affine Lefschetz theorem, we first reduce to the case of a curve as a basis and then exploit the classical Cassels-Tate pairing.

\paragraph{Reduction to the curve case.}
Let $Y \hookrightarrow X$ be an ample smooth geometrically connected hypersurface section (this exists by~\cite{PoonenBertini}, Proposition~2.7) with (necessarily) affine complement $U \hookrightarrow X$.  Base changing to $\bar{k}$ and writing $\bar{X} = X \times_k \bar{k}$ etc., one has by~\cite{MilneÉtaleCohomology}, p.~94, Remark~III.1.30 a long exact sequence
\begin{align} \label{eq:longexactsequence}
    \ldots \to \H^i_c(\bar{U}, \Acal[\ell^n]) \to \H^i(\bar{X}, \Acal[\ell^n]) \to \H^i(\bar{Y}, \Acal[\ell^n]) \to \H^{i+1}_c(\bar{U}, \Acal[\ell^n]) \to \ldots
\end{align}
(Note that $\H^i_c(\bar{X}, \Fcal) = \H^i(\bar{X}, \Fcal)$ since $\bar{X}$ is proper, and likewise for $\bar{Y}$.)

Since $\Acal[\ell^n]/X$ is étale, Poincaré duality~\cite{MilneÉtaleCohomology}, p.~276, Corollary~VI.11.2 gives us
\[
    \H^i_c(\bar{U},\Acal[\ell^n]) = \H^{2d-i}(\bar{U},(\Acal[\ell^n])^\vee(d)).
\]
(Note that the varieties live over a separably closed field.)  By the affine Lefschetz theorem~\cite{MilneÉtaleCohomology}, p.~253, Theorem~VI.7.2, one has $\H^{2d-i}(\bar{U},(\Acal[\ell^n])^\vee(d)) = 0$ for $2d-i > d$, i.\,e.\ for $i < d$.  Analogously, $\H^{i+1}_c(\bar{U}, \Acal[\ell^n]) = 0$ for $i+1 < d$.  Plugging this into~\eqref{eq:longexactsequence}, one gets an isomorphism
\[
    \H^i(\bar{X}, \Acal[\ell^n]) \isoto \H^i(\bar{Y}, \Acal[\ell^n])
\]
for $i+1 < d$. Inductively, it follows that the cohomology groups in dimension $0$ and $1$ are isomorphic to the cohomology groups of a \emph{surface} $S$.

If $Y = C \hookrightarrow X$ is a \emph{curve}, one gets at least
\begin{align}
    \H^0(\bar{X}, \Acal[\ell^n]) \isoto \H^0(\bar{C}, \Acal[\ell^n]), \label{eq:H0XC} \\
    \H^1(\bar{X}, \Acal[\ell^n]) \hookrightarrow \H^i(\bar{C}, \Acal[\ell^n]). \label{eq:H1XC}
\end{align}

There are short exact sequences
\[
    0 \to \H^{i-1}(\bar{X},\Acal[\ell^n])_\Gamma \to \H^i(X, \Acal[\ell^n]) \to \H^i(\bar{X}, \Acal[\ell^n])^\Gamma \to 0.
\]
Using~\eqref{eq:H0XC}, one gets a commutative diagram
\[\xymatrix{
    \H^0(X, \Acal[\ell^n]) \ar[r]^\iso \ar[d] & \H^0(\bar{X}, \Acal[\ell^n])^\Gamma \ar[d]^\iso \\
    \H^0(C, \Acal[\ell^n]) \ar[r]^\iso & \H^0(\bar{C}, \Acal[\ell^n])^\Gamma,
}\]
and hence an isomorphism $\H^0(X, \Acal[\ell^n]) \isoto \H^0(C, \Acal[\ell^n])$.

This implies
\begin{equation} \label{eq:TorEqual}
    \Tors\Acal(X)[\ell^\infty] = \H^0(X,\Acal[\ell^\infty]) = \H^0(C,\Acal[\ell^\infty]) = \Tors\Acal(C)[\ell^\infty].
\end{equation}

Using~\eqref{eq:H0XC} and~\eqref{eq:H1XC}, one gets a commutative diagram
\[\xymatrix{
    0 \ar[r] & \H^0(\bar{X},\Acal[\ell^n])_\Gamma \ar[d]^\iso \ar[r] & \H^1(X, \Acal[\ell^n]) \ar@{^{(}-->}[d] \ar[r] & \H^1(\bar{X}, \Acal[\ell^n])^\Gamma \ar@{^{(}->}[d] \ar[r] & 0 \\
    0 \ar[r] & \H^0(\bar{C},\Acal[\ell^n])_\Gamma \ar[r] & \H^1(C, \Acal[\ell^n]) \ar[r] & \H^1(\bar{C}, \Acal[\ell^n])^\Gamma \ar[r] & 0.
}\]
and hence, by the snake lemma, an injection $\H^1(X, \Acal[\ell^n]) \hookrightarrow \H^1(C, \Acal[\ell^n])$.

\begin{remark} \label[remark]{rem:rkEqual}
By the snake lemma, this injection is an isomorphism iff $\H^1(\bar{X}, \Acal[\ell^n])^\Gamma \hookrightarrow \H^1(\bar{C}, \Acal[\ell^n])^\Gamma$ is surjective, e.\,g.\ $\dim{C} \geq 2$.  Under this condition, one gets (by passing to the inverse limit $\varprojlim_n$ and tensoring with $\Q_\ell$) the equality of the $L$-functions $L(\Acal/X,s) = L(\Acal/C,s)$, and hence, under assumption of the Birch-Swinnerton-Dyer conjecture ($\Sha(\Acal/C)[\ell^\infty]$ finite), the equality of the ranks $\rk_\Z\Acal(X) = \rk_\Z\Acal(C)$.
\end{remark}

We want the map $\H^1(X,\Acal)[\ell^n] \to \H^1(C,\Acal)[\ell^n]$ of the $\ell^n$-torsion of the Tate-Shafarevich groups to be injective in order for the generalised Cassels-Tate pairing to be non-degenerate (at least on the non-divisible part).

\[\xymatrix{
    0 \ar[r] &\Acal(X)/\ell^n \ar@{^{(}-->}[d] \ar[r] &\H^1(X,\Acal[\ell^n]) \ar@{^{(}->}[d] \ar[r] &\H^1(X,\Acal)[\ell^n] \ar[d]\ar[r] &0 \\
    0 \ar[r] &\Acal(C)/\ell^n \ar[r] &\H^1(C,\Acal[\ell^n]) \ar[r] &\H^1(C,\Acal)[\ell^n] \ar[r] &0
}\]
By the snake lemma, the left vertical map is an injection and $\H^1(X,\Acal)[\ell^n] \to \H^1(C,\Acal)[\ell^n]$ is injective if this injection $\Acal(X)/\ell^n \hookrightarrow \Acal(C)/\ell^n$ is surjective.

\begin{remark}
Note that for $\dim{Y} \geq 2$ and $Y \hookrightarrow X$ excised by a sequence of ample hypersurface sections, $\rk\Acal(X) = \rk\Acal(Y)$ by~\cref{rem:rkEqual}.
\end{remark}

\paragraph{The curve case.}
By~\cite{MilneADT}, p.~176, (a), one has for $C/k$ a smooth proper geometrically connected curve over the finite ground field $k$
\[
    \H^3(C,\G_m) = \Q/\Z.
\]

By the local-to-global spectral sequence $\H^p(X, \EExt_X^q(\Acal,\G_m)) \Rightarrow \Ext_X^{p+q}(\Acal,\G_m)$, and using $\HHom_X(\Acal,\G_m) = 0$ and the Barsotti-Weil formula $\EExt_X^1(\Acal,\G_m) = \Acal^\vee$, we get edge morphisms 
\[
    \H^r(X,\Acal^\vee) \to \Ext_X^{r+1}(\Acal,\G_m),
\]
which are injective for $r = 1$.  Composing this for $r = 1$ with the Yoneda pairing for $r = 2$
\[
    \Ext_C^r(\Acal,\G_m) \times \H^{3-r}(C,\Acal) \to \H^3(C,\G_m) \isoto \Q/\Z,
\]
induces a pairing with left and right kernels the divisible part
\[
    \H^1(C,\Acal^\vee) \times \H^{1}(C,\Acal) \stackrel{\cup}{\to} \H^{3}(C,\G_m) = \Q/\Z.
\]
This is the Cassels-Tate pairing, see~\cite{MilneADT}, p.~199--203.
\end{proof}

\begin{remark}
The homomorphism $\Acal(X)/\ell^n \to \Acal(C)/\ell^n$ is e.\,g.\ surjective if $\rk\Acal(C) = 0$ since $\Tors\Acal(X)[\ell^n] = \Tors\Acal(C)[\ell^n]$ by~\eqref{eq:TorEqual} and $\Acal(X)/\ell^n \hookrightarrow \Acal(C)/\ell^n$ is injective.  For example, one has $\rk\Acal(C) = 0$ if $\Acal/C$ is constant and $C \cong \mathbf{P}_k^1$:  The rank of the Mordell-Weil group of a constant Abelian variety over a projective space has rank $0$, since there are no non-constant $k$-morphisms $\mathbf{P}^n_k \to A$, see~\cite{MilneAbelianVarieties}, p.~107, Corollary~3.9.
\end{remark}

\paragraph{Acknowledgements.}
I thank Uwe Jannsen, Jean-Louis Colliot-Thélène, Brian Conrad, Peter Jossen, Moritz Kerz, Niko Naumann, Jakob Stix, and, from mathoverflow, abx, Angelo, anon, Martin Bright, Kestutis Cesnavicius, Torsten Ekedahl, Laurent Moret-Bailly, ulrich and xuhan; for proofreading Patrick Forré, Peter Jossen, Niko Naumann and Antonella Perucca; finally the Studienstiftung des deutschen Volkes for financial and ideational support.

\phantomsection 
\addcontentsline{toc}{section}{\refname}
{\small
	\bibliographystyle{tkalpha3}
	\bibliography{Dissertation}

\newcommand{\etalchar}[1]{$^{#1}$}
\begin{thebibliography}{EGAIV${}_3$d}
\providecommand{\url}[1]{\url{#1}}
\providecommand{\urlprefix}{\acronym{URL}~}
\expandafter\ifx\csname urlstyle\endcsname\relax
  \providecommand{\doi}[1]{doi:\discretionary{}{}{}#1}\else
  \providecommand{\doi}{doi:\discretionary{}{}{}\begingroup
  \urlstyle{rm}\Url}\fi
\providecommand{\bibAnnoteFile}[1]{%
  \IfFileExists{#1}{\begin{quotation}\noindent\textsc{Key:} #1\\
  \textsc{Annotation:}\ \input{#1}\end{quotation}}{}}
\providecommand{\bibAnnote}[2]{%
  \begin{quotation}\noindent\textsc{Key:} #1\\
  \textsc{Annotation:}\ #2\end{quotation}}
\providecommand{\bibinfo}[2]{#2}
\providecommand{\eprint}[2][]{\url{#2}}

\bibitem[Art69]{ArtinApproximation}
\bibinfo{author}{\textsc{Artin}, Michael}:
  {\selectlanguage{english}\emph{\bibinfo{title}{{Algebraic approximation of
  structures over complete local rings.}}}} In: \bibinfo{journal}{Publ.\ Math.\
  IHÉS}, \textbf{\bibinfo{volume}{36}} (\bibinfo{year}{1969}),
  \bibinfo{pages}{23--58}.
\bibAnnoteFile{ArtinApproximation}

\bibitem[BLR90]{BLR}
\bibinfo{author}{\textsc{Bosch}, Siegfried};
  \bibinfo{author}{\textsc{Lütkebohmert}, Werner} and
  \bibinfo{author}{\textsc{Raynaud}, Michel}:
  {\selectlanguage{english}\emph{\bibinfo{title}{{Néron models.}}}}
  \bibinfo{publisher}{{Ergebnisse der Mathematik und ihrer Grenzgebiete, 3.
  Folge, \textbf{21}. Berlin etc.: Springer-Verlag. x, 325 p. }}
  \bibinfo{year}{1990}.
\bibAnnoteFile{BLR}

\bibitem[Bos03]{Bosch}
\bibinfo{author}{\textsc{Bosch}, Siegfried}:
  {\selectlanguage{ngerman}\emph{\bibinfo{title}{Algebra}}}.
  \bibinfo{edition}{fünfte Auf\/lage}, \bibinfo{publisher}{Springer-Lehrbuch.
  Berlin: Springer} \bibinfo{year}{2003}.
\bibAnnoteFile{Bosch}

\bibitem[Con]{Conrad-pic}
\bibinfo{author}{\textsc{Conrad}, Brian}:
  {\selectlanguage{english}\emph{\bibinfo{title}{{Math 248B. Picard functors
  for curves}}}}.
  \urlprefix\url{http://math.stanford.edu/~conrad/248BPage/handouts/pic.pdf}.
\bibAnnoteFile{Conrad-pic}

\bibitem[{Den}88]{DeningerProperBaseChange}
\bibinfo{author}{\textsc{{Deninger}}, Christopher}:
  {\selectlanguage{english}\emph{\bibinfo{title}{{A proper base change theorem
  for non-torsion sheaves in étale cohomology.}}}} In: \bibinfo{journal}{{J.
  Pure Appl. Algebra}}, \textbf{\bibinfo{volume}{50}}(\bibinfo{number}{3})
  (\bibinfo{year}{1988}), \bibinfo{pages}{231--235}.
\bibAnnoteFile{DeningerProperBaseChange}

\bibitem[dJ]{GabberBrauerGroup}
\bibinfo{author}{\textsc{de~Jong}, Aise~Johan}:
  {\selectlanguage{english}\emph{\bibinfo{title}{{A result of Gabber.}}}}
  \urlprefix\url{www.math.columbia.edu/~dejong/papers/2-gabber.pdf}.
\bibAnnoteFile{GabberBrauerGroup}

\bibitem[Eis95]{Eisenbud}
\bibinfo{author}{\textsc{Eisenbud}, David}:
  {\selectlanguage{english}\emph{\bibinfo{title}{{Commutative algebra. With a
  view toward algebraic geometry.}}}} \bibinfo{publisher}{{Graduate Texts in
  Mathematics. \textbf{150}. Berlin: Springer-Verlag. xvi, 785 p.}}
  \bibinfo{year}{1995}.
\bibAnnoteFile{Eisenbud}

\bibitem[FGI{\etalchar{+}}05]{FGAExplained}
\bibinfo{author}{\textsc{Fantechi}, Barbara};
  \bibinfo{author}{\textsc{Göttsche}, Lothar};
  \bibinfo{author}{\textsc{Illusie}, Luc} et~al.:
  {\selectlanguage{english}\emph{\bibinfo{title}{{Fundamental algebraic
  geometry: Grothendieck's FGA explained.}}}} \bibinfo{publisher}{{Mathematical
  Surveys and Monographs \textbf{123}. Providence, RI: American Mathematical
  Society (AMS). x, 339~p. }} \bibinfo{year}{2005}.
\bibAnnoteFile{FGAExplained}

\bibitem[Fu11]{LeiFuÉtaleCohomologyTheory}
\bibinfo{author}{\textsc{Fu}, Lei}:
  {\selectlanguage{english}\emph{\bibinfo{title}{{Étale cohomology theory.}}}}
  \bibinfo{publisher}{{Nankai Tracts in Mathematics \textbf{13}. Hackensack,
  NJ: World Scientific. ix, 611~p.}} \bibinfo{year}{2011}.
\bibAnnoteFile{LeiFuÉtaleCohomologyTheory}

\bibitem[Gro68]{GroupeDeBrauerIII}
---{}---{}--- {\selectlanguage{french}\emph{\bibinfo{title}{{Le groupe de
  Brauer. III\,: Exemples et complements.}}}} \bibinfo{howpublished}{{Dix
  Exposes Cohomologie Schemas, Advanced Studies Pure Math. \textbf{3},
  88--188.}} \bibinfo{year}{1968}.
\bibAnnoteFile{GroupeDeBrauerIII}

\bibitem[GW10]{Görtz-Wedhorn}
\bibinfo{author}{\textsc{Görtz}, Ulrich} and \bibinfo{author}{\textsc{Wedhorn},
  Torsten}: {\selectlanguage{english}\emph{\bibinfo{title}{{Algebraic geometry
  I. Schemes. With examples and exercises.}}}} \bibinfo{publisher}{{Advanced
  Lectures in Mathematics. Wiesbaden: Vieweg+Teubner. vii, 615~p.}}
  \bibinfo{year}{2010}.
\bibAnnoteFile{Görtz-Wedhorn}

\bibitem[Har83]{HartshorneAG}
\bibinfo{author}{\textsc{Hartshorne}, Robin}:
  {\selectlanguage{english}\emph{\bibinfo{title}{{Algebraic geometry. Corr.\
  3rd printing.}}}} \bibinfo{publisher}{{Graduate Texts in Mathematics,
  \textbf{52}. New York-Heidelberg-Berlin: Springer-Verlag. XVI, 496 p.}}
  \bibinfo{year}{1983}.
\bibAnnoteFile{HartshorneAG}

\bibitem[{Kah}06]{KahnGroupeDesClasses}
\bibinfo{author}{\textsc{{Kahn}}, Bruno}:
  {\selectlanguage{french}\emph{\bibinfo{title}{{Sur le groupe des classes d'un
  sch\'ema arithm\'etique. (Avec un appendice de Marc Hindry).}}}} In:
  \bibinfo{journal}{{Bull.\ Soc.\ Math.\ Fr.}},
  \textbf{\bibinfo{volume}{134}}(\bibinfo{number}{3}) (\bibinfo{year}{2006}),
  \bibinfo{pages}{395--415}.
\bibAnnoteFile{KahnGroupeDesClasses}

\bibitem[Kel14${}_2$]{KellerHeight}
\bibinfo{author}{\textsc{Keller}, Timo}:
  {\selectlanguage{english}\emph{\bibinfo{title}{On the conjecture of Birch and
  Swinnerton-Dyer for Abelian schemes over higher dimensional bases over finite
  fields}}}. Preprint, 2014.
\bibAnnoteFile{KellerHeight}

\bibitem[Lic83]{Lichtenbaums=1}
\bibinfo{author}{\textsc{Lichtenbaum}, Stephen}:
  {\selectlanguage{english}\emph{\bibinfo{title}{{Zeta-functions of varieties
  over finite fields at $s=1$.}}}} \bibinfo{howpublished}{{Arithmetic and
  geometry, Pap.\ dedic.\ I.\ R.\ Shafarevich, Vol.\ I: Arithmetic, Prog.\
  Math.\ \textbf{35}, 173--194.}} \bibinfo{year}{1983}.
\bibAnnoteFile{Lichtenbaums=1}

\bibitem[Liu06]{Liu2006}
\bibinfo{author}{\textsc{Liu}, Qing}:
  {\selectlanguage{english}\emph{\bibinfo{title}{{Algebraic geometry and
  arithmetic curves. Transl. by Reinie Ern\'e.}}}} \bibinfo{publisher}{{Oxford
  Graduate Texts in Mathematics \textbf{6}. Oxford: Oxford University Press.
  xv, 577~p. }} \bibinfo{year}{2006}.
\bibAnnoteFile{Liu2006}

\bibitem[Mat86]{Matsumura}
\bibinfo{author}{\textsc{Matsumura}, Hideyuki}:
  {\selectlanguage{english}\emph{\bibinfo{title}{Commutative ring theory}}}.
  Cambridge Studies in Advanced Mathematics, \textbf{8},
  \bibinfo{publisher}{Cambridge University Press} \bibinfo{year}{1986}.
\bibAnnoteFile{Matsumura}

\bibitem[Mil80]{MilneÉtaleCohomology}
\bibinfo{author}{\textsc{Milne}, James~S.}:
  {\selectlanguage{english}\emph{\bibinfo{title}{{Étale cohomology.}}}}
  \bibinfo{publisher}{{Princeton Mathematical Series. \textbf{33}. Princeton,
  New Jersey: Princeton University Press. XIII, 323 p.}} \bibinfo{year}{1980}.
\bibAnnoteFile{MilneÉtaleCohomology}

\bibitem[Mil86a]{MilneAbelianVarieties}
---{}---{}--- {\selectlanguage{english}\emph{\bibinfo{title}{{Abelian
  varieties.}}}} \bibinfo{howpublished}{{Arithmetic geometry, Pap. Conf.,
  Storrs/Conn. 1984, 103--150 (1986).}} \bibinfo{year}{1986}.
\bibAnnoteFile{MilneAbelianVarieties}

\bibitem[Mil86b]{MilneADT}
---{}---{}--- {\selectlanguage{english}\emph{\bibinfo{title}{{Arithmetic
  duality theorems.}}}} \bibinfo{publisher}{{Perspectives in Mathematics, Vol.
  1. Boston etc.: Academic Press. Inc. Harcourt Brace Jovanovich, Publishers.
  X, 421 p.}} \bibinfo{year}{1986}.
\bibAnnoteFile{MilneADT}

\bibitem[Mum70]{MumfordAbelianVarieties}
\bibinfo{author}{\textsc{Mumford}, David}:
  {\selectlanguage{english}\emph{\bibinfo{title}{{Abelian varieties}}}}.
  \bibinfo{publisher}{Oxford University Press. viii, 242 p.}
  \bibinfo{year}{1970}.
\bibAnnoteFile{MumfordAbelianVarieties}

\bibitem[{Poo}05]{PoonenBertini}
\bibinfo{author}{\textsc{{Poonen}}, Bjorn}:
  {\selectlanguage{english}\emph{\bibinfo{title}{{Bertini theorems over finite
  fields.}}}} In: \bibinfo{journal}{{Ann.\ Math.}},
  \textbf{\bibinfo{volume}{160}}(\bibinfo{number}{3}) (\bibinfo{year}{2005}),
  \bibinfo{pages}{1099--1127}.
\bibAnnoteFile{PoonenBertini}

\bibitem[Ser02]{SerreGaloisCohomology}
\bibinfo{author}{\textsc{Serre}, Jean-Pierre}:
  {\selectlanguage{english}\emph{\bibinfo{title}{{Galois cohomology. Transl.\
  from the French by Patrick Ion. 2nd printing.}}}}
  \bibinfo{publisher}{{Berlin: Springer}} \bibinfo{year}{2002}.
\bibAnnoteFile{SerreGaloisCohomology}

\bibitem[Tat66]{Artin-Tate}
\bibinfo{author}{\textsc{Tate}, John~T.}:
  {\selectlanguage{english}\emph{\bibinfo{title}{{On the conjectures of Birch
  and Swinnerton-Dyer and a geometric analog.}}}} \bibinfo{howpublished}{{Dix
  Expos\'es Cohomologie Sch\'emas, Advanced Studies Pure Math. \textbf{3},
  189--214 (1968); S\'em.\ Bourbaki 1965/66, Exp.\ No.\ 306, 415--440.}}
  \bibinfo{year}{1966}.
\bibAnnoteFile{Artin-Tate}

\bibitem[Wei97]{Weibel}
\bibinfo{author}{\textsc{Weibel}, Charles~A.}:
  {\selectlanguage{english}\emph{\bibinfo{title}{An introduction to homological
  algebra}}}. Cambridge studies in advanced mathematics~\textbf{38},
  \bibinfo{publisher}{Cambridge university press} \bibinfo{year}{1997}.
\bibAnnoteFile{Weibel}

\bibitem[EGAII]{EGAII}
\bibinfo{author}{\textsc{Grothendieck}, Alexandre} and
  \bibinfo{author}{\textsc{Dieudonné}, Jean}:
  {\selectlanguage{french}\emph{\bibinfo{title}{{Éléments de
  géométrie algébrique\,: II. Étude globale élémentaire de quelques classes de
  morphismes.}}}} In: \bibinfo{journal}{{Publ.\ Math.\ IHÉS}},
  \textbf{\bibinfo{volume}{8}} (\bibinfo{year}{1961}), \bibinfo{pages}{5--222}.
\bibAnnoteFile{EGAII}

\bibitem[EGAIII${}_1$]{EGAIII1}
---{}---{}--- {\selectlanguage{french}\emph{\bibinfo{title}{{Eléments de géométrie
  algébrique\,: III. Étude cohomologique des faisceaux cohérents, Première
  partie.}}}} In: \bibinfo{journal}{{Publ.\ Math.\ IHÉS}},
  \textbf{\bibinfo{volume}{11}} (\bibinfo{year}{1961}),
  \bibinfo{pages}{5--167}.
\bibAnnoteFile{EGAIII1}

\bibitem[EGAIII${}_2$]{EGAIII2}
---{}---{}--- {\selectlanguage{french}\emph{\bibinfo{title}{{Éléments de
  géométrie algébrique\,: III. Étude cohomologique des faisceaux cohérents,
  Seconde partie.}}}} In: \bibinfo{journal}{{Publ.\ Math.\ IHÉS}},
  \textbf{\bibinfo{volume}{17}} (\bibinfo{year}{1963}), \bibinfo{pages}{5--91}.
\bibAnnoteFile{EGAIII2}

\bibitem[EGAIV${}_3$]{EGAIV3}
---{}---{}--- {\selectlanguage{french}\emph{\bibinfo{title}{{Éléments de
  géométrie algébrique\,: IV. Étude locale des schémas et des morphismes de
  schémas, Troisième partie.}}}} In: \bibinfo{journal}{{Publ.\ Math.\ IHÉS}},
  \textbf{\bibinfo{volume}{28}} (\bibinfo{year}{1966}),
  \bibinfo{pages}{5--255}.
\bibAnnoteFile{EGAIV3}

\bibitem[EGAIV${}_4$]{EGAIV4}
---{}---{}--- {\selectlanguage{french}\emph{\bibinfo{title}{{Éléments de
  géométrie algébrique\,: IV. Étude locale des schémas et des morphismes de
  schémas, Quatrième partie.}}}} In: \bibinfo{journal}{{Publ.\ Math.\ IHÉS}},
  \textbf{\bibinfo{volume}{32}} (\bibinfo{year}{1967}),
  \bibinfo{pages}{5--361}.
\bibAnnoteFile{EGAIV4}

\bibitem[SGA1]{SGA1}
\bibinfo{author}{\textsc{Grothendieck}, Alexandre}:
  {\selectlanguage{french}\emph{\bibinfo{title}{{Séminaire de Géométrie
  Algébrique du Bois Marie -- 1960--61 -- Revêtements étales et groupe
  fondamental -- (SGA 1)}}}}. (Lecture notes in mathematics \textbf{224}),
  \bibinfo{publisher}{Berlin; New York: Springer-Verlag. xxii+447}
  \bibinfo{year}{1971}.
\bibAnnoteFile{SGA1}

\bibitem[SGA4.2]{SGA42}
\bibinfo{author}{\textsc{Artin}, Michael};
  \bibinfo{author}{\textsc{Grothendieck}, Alexandre} and
  \bibinfo{author}{\textsc{Verdier}, Jean-Louis}:
  {\selectlanguage{french}\emph{\bibinfo{title}{{Séminaire de Géométrie
  Algébrique du Bois Marie -- 1963--64 -- Théorie des topos et cohomologie
  étale des schémas -- (SGA 4) -- vol.\ 2}}}}. Lecture notes in mathematics (in
  French) \textbf{270}, \bibinfo{publisher}{Berlin; New York: Springer-Verlag.
  iv+418} \bibinfo{year}{1972}.
\bibAnnoteFile{SGA42}

\bibitem[SGA4.3]{SGA43}
---{}---{}--- {\selectlanguage{french}\emph{\bibinfo{title}{{Séminaire de
  Géométrie Algébrique du Bois Marie -- 1963--64 -- Théorie des topos et
  cohomologie étale des schémas -- (SGA 4) -- vol.\ 3}}}}. Lecture notes in
  mathematics (in French) \textbf{305}, \bibinfo{publisher}{Berlin; New York:
  Springer-Verlag. vi+640} \bibinfo{year}{1972}.
\bibAnnoteFile{SGA43}

\bibitem[SGA6]{SGA6}
\bibinfo{author}{\textsc{Berthelot}, Pierre};
  \bibinfo{author}{\textsc{Grothendieck}, Alexandre} and
  \bibinfo{author}{\textsc{Illusie}, Luc}:
  {\selectlanguage{french}\emph{\bibinfo{title}{{Séminaire de Géométrie
  Algébrique du Bois Marie -- 1966--67 -- Théorie des intersections et théorème
  de Riemann-Roch -- (SGA 6)}}}}. Lecture notes in mathematics (in French)
  \textbf{225}, \bibinfo{publisher}{Berlin; New York: Springer-Verlag. xii+700}
  \bibinfo{year}{1971}.
\bibAnnoteFile{SGA6}

\end{thebibliography}
}

{\small \textsc{Timo Keller, Fakultät für Mathematik, Universität Regensburg, 93040 Regensburg, Germany}}

{\small \emph{E-Mail address: timo.keller@mathematik.uni-regensburg.de}}

\end{document}